\documentclass[twoside,11pt,reqno]{amsart}
\usepackage{amsmath,amssymb,amscd,mathrsfs,wasysym}

\newtheorem{Theorem}{Theorem}[section]
\newtheorem{Corollary}[Theorem]{Corollary}
\newtheorem{Lemma}[Theorem]{Lemma}

\theoremstyle{remark}
\newtheorem{Remark}[Theorem]{Remark}

\numberwithin{equation}{section}

\def\ev{\operatorname{ev}}

\def\op{{\operatorname{op}}}

\def\Tw{\operatorname{T}}

\def\cE{e}
\def\cEx{E}
\def\cF{f}

\def\cM{M}
\def\cD{D}
\def\cX{X}
\def\cZ{Z}
\def\cP{P}
\def\cY{Y}
\def\cL{L}
\def\cS{S}
\def\cT{T}
\def\qH{\mathscr{H}}

\def\rep#1{\operatorname{Rep}(#1)}
\def\proj#1{\operatorname{Proj}(#1)}

\def\id{\operatorname{id}}

\def\ev{{\operatorname{ev}}}

\def\C{{\mathbb C}}
\def\Q{{\mathbb Q}}
\def\Z{{\mathbb Z}}

\def\0{{\bar 0}}
\def\1{{\bar 1}}

\def\A{{\mathscr A}}

\def\hom{{\operatorname{Hom}}}

\def\End{{\operatorname{End}}}

\def\Col{{\operatorname{Col}}}

\def\Std{{\operatorname{Std}}}
\def\RStd{{\operatorname{Rev}}}

\def\ch{{\operatorname{ch}\:}}
\def\wt{{\operatorname{wt}}}
\def\height{{\operatorname{ht}}}
\def\inf{{\operatorname{inf}}}

\def\bi{\text{\boldmath$i$}}

\def\bm{\text{\boldmath$m$}}
\def\eps{{\varepsilon}}
\def\phi{{\varphi}}

\def\la{{\lambda}}
\def\La{{\Lambda}}

\def\al{{\alpha}}

\def\bpar{\la}

\def\rpar{\la}
\def\underbar{\mathpalette\@underbar}
\def\@underbar#1#2{\settowidth{\@tempdimb}{$#1#2$}\@tempdimb=0.8\@tempdimb
                   \ooalign{$#1#2$\crcr%
                         \hfil\rule[-.5mm]{\@tempdimb}{.4pt}\hfil}}

\newdimen\hoogte    \hoogte=12.5pt
\newdimen\breedte   \breedte=12.5pt
\newdimen\dikte     \dikte=0.5pt

\newenvironment{young}{\begingroup
       \def\vr{\vrule height0.8\hoogte width\dikte depth 0.2\hoogte}
       \def\fbox##1{\vbox{\offinterlineskip
                    \hrule height\dikte
                    \hbox to \breedte{\vr\hfill##1\hfill\vr}
                    \hrule height\dikte}}
       \vbox\bgroup \offinterlineskip \tabskip=-\dikte \lineskip=-\dikte
            \halign\bgroup &\fbox{##\unskip}\unskip  \crcr }
       {\egroup\egroup\endgroup}
\def\diagram#1{\relax\ifmmode\vcenter{\,\begin{young}#1\end{young}\,}\else%
              $\vcenter{\,\begin{young}#1\end{young}\,}$\fi}

\begin{document}

\title[Ariki's categorification theorem]{The degenerate analogue of Ariki's categorification theorem}
\author{Jonathan Brundan}
\address{Department of Mathematics\\University of Oregon\\Eugene OR 97403\\USA}
\email{brundan@uoregon.edu}
\author{Alexander Kleshchev}
\address{Department of Mathematics\\University of Oregon\\Eugene OR 97403\\USA}
\email{klesh@uoregon.edu}
\thanks{Research supported in part by NSF grant no. DMS-0654147.}

\begin{abstract}
We explain how to deduce the degenerate analogue of
Ariki's categorification theorem
over the ground field $\C$
as an application of 
Schur-Weyl duality for higher levels
and the Kazhdan-Lusztig conjecture in finite type A.
We also discuss some supplementary topics, including
Young modules, tensoring with sign, tilting modules and Ringel duality.
\end{abstract}
\maketitle

\section{Introduction}

Ariki's categorification theorem proved in \cite{Ari}
is concerned
with the representation theory of 
cyclotomic Hecke algebras 
of type $G(l,1,d)$.
For $l=1$, the result was conjectured earlier
by Lascoux, Leclerc and Thibon \cite{LLT}, and
a proof was also announced (but never published) 
by Grojnowski following \cite{Gnote}.

For $d \geq 0$ and a
dominant integral weight $\La$
of level $l$ for the affine Lie algebra $\widehat{\mathfrak{sl}}_e(\C)$,
there is an associated cyclotomic Hecke algebra
$\qH_d^\La$, namely, the quotient of the 
affine Hecke algebra $\qH_d$
associated to the general linear group $GL_d$
by the two-sided ideal generated by the element
$$
(X_1 - \xi^{m_1}) \cdots (X_1 - \xi^{m_l}).
$$
Here, the defining parameter $\xi$ of the Hecke algebra $\qH_d$
should be taken to be a primitive complex $e$th root of unity,
$m_1,\dots,m_l \in \Z / e\Z$
are defined by expanding $\La=\La_{m_1}+\cdots + \La_{m_l}$
in terms of fundamental dominant
weights, and we are using the
Bernstein presentation for $\qH_d$.

The categorification theorem asserts that the direct sum
$$
\bigoplus_{d \geq 0} [\proj{\qH_d^\La}]
$$
of the Grothendieck groups of the categories of 
finitely generated projective modules over these algebras
can be identified
with the standard $\Z$-form 
$V(\La)_\Z$
for the irreducible highest weight
$\widehat{\mathfrak{sl}}_e(\C)$-module of highest weight $\La$.
Under the identification, the basis for the Grothendieck group arising from
projective indecomposable 
modules corresponds to Lusztig's canonical basis,
and the actions of the Chevalley generators of $\widehat{\mathfrak{sl}}_e(\C)$
correspond to certain $i$-induction and $i$-restriction functors.
Moreover, there is another natural family of
$\qH_d^\La$-modules,
the so-called
{\em Specht modules}, such that
the {decomposition matrix} describing 
composition multiplicities of Specht modules is the
transpose of the matrix describing the
expansion of the canonical basis 
for $V(\La)_\Z$ in terms of
the {monomial basis} of an appropriate level
$l$ Fock space $F(\La)_\Z$ containing $V(\La)_\Z$ as a submodule.

All of this also makes sense if $\xi$ is generic (not a root of unity),
replacing $\widehat{\mathfrak{sl}}_e(\C)$ 
with $\mathfrak{gl}_\infty(\C)$ so that
$m_1,\dots,m_l$ are ordinary integers
rather than integers modulo $e$.
Even in the generic case, the proof
is highly non-trivial (in levels bigger than one),
depending ultimately on the 
$p$-adic analogue of the Kazhdan-Lusztig conjecture formulated in
\cite{Z1},
which is a special case of the Deligne-Langlands conjecture proved in 
\cite{KL2} (see also \cite[Chapter 8]{CG}), 
as well as Lusztig's geometric construction of canonical
bases from \cite{Lu0}.

This article is concerned with the
analogue of Ariki's categorification theorem for the {\em degenerate
cyclotomic Hecke algebra} $H^\La_d$,
namely,
the quotient of the degenerate affine Hecke algebra $H_d$
from \cite{D} by the two-sided ideal generated by the element
$$
(x_1-m_1)\cdots (x_1-m_l).
$$
Here, $m_1,\dots,m_l \in \Z$ arise as before by 
writing the dominant integral weight $\La$ for $\mathfrak{gl}_\infty(\C)$
as $\La = \La_{m_1}+\cdots+\La_{m_l}$.
Remarkably, by \cite[Corollary 2]{young}, the degenerate cyclotomic Hecke
algebra $H^\La_d$ is 
{\em isomorphic} to the non-degenerate cyclotomic Hecke 
algebra $\qH^\La_d$
for generic $\xi$.
Combining this isomorphism 
with the cellular algebra structure on $H^\La_d$ from \cite[$\S$6]{AMR},
which provides an intrinsic notion of Specht module in the degenerate
setting, the degenerate analogue of Ariki's categorification theorem
follows almost at once from Ariki's theorem 
for generic $\xi$ (and vice versa).
It should also be possible 
to give a direct proof of the degenerate analogue 
without appealing to the isomorphism
$H^\La_d \cong \qH^\La_d$ by
using Lusztig's geometric approach to the representation theory of $H_d$
from  \cite{Lc, Ld} to obtain the key multiplicity formula
for standard modules in terms of intersection cohomology.

The goal in the remainder of the article is to explain a completely
different way to prove
the degenerate analogue of Ariki's categorification
theorem, based instead on the Schur-Weyl duality for higher levels
developed in \cite{schur}.
This duality is a generalization of classical Schur-Weyl duality,
in which the group algebra of the symmetric group
gets replaced by the
degenerate cyclotomic Hecke algebra $H_d^\La$,
and the category of polynomial representations of the
general linear group gets is replaced by certain
blocks of a parabolic analogue of the Bernstein-Gelfand-Gelfand
category $\mathcal O$ from \cite{BGG}. 
A great deal is known about parabolic
category $\mathcal O$, thanks in part to 
the Kazhdan-Lusztig conjecture formulated in \cite{KL1} that
was proved in
\cite{BB, BrK}.
Pushing this detailed information through Schur-Weyl duality, we recover
in a very tidy way almost all known results about the 
representation theory
of the degenerate cyclotomic Hecke algebras over the ground field $\C$,
including the desired categorification theorem.

This means that there are two quite different
ways to describe composition multiplicities of Specht modules
in the degenerate case, 
one in terms of intersection cohomology of closures of nilpotent orbits
of linear quivers, and the other in terms of
parabolic Kazhdan-Lusztig polynomials associated to the symmetric group.
The same coincidence has been observed before at the affine level:
in \cite{AS} Arakawa and Suzuki
explained how to express
the multiplicities of standard modules 
of degenerate affine Hecke algebras 
in terms of finite type A Kazhdan-Lusztig polynomials
by means of another variation on Schur-Weyl duality involving
category $\mathcal O$.
The geometric 
explanation behind all these coincidences comes from a classic result of
Zelevinsky \cite{Zel}; we refer the reader to a recent paper of
Henderson \cite{H} for a comprehensive account (and extensions to the
root of unity case).

In the Schur-Weyl duality approach,
parabolic category $\mathcal O$ fits in very nicely into the
categorification picture: its Grothendieck group plays the
role of the level $l$ Fock space 
$F(\La)_\Z$ mentioned above
and the embedding of $V(\La)_\Z$ into $F(\La)_\Z$ corresponds
at the level of categories
to a ``Schur functor''
which maps projective $H_d^\La$-modules to 
projective modules in parabolic category $\mathcal O$.
This can be regarded as the extension
to higher levels in the degenerate case
of a conjecture of Leclerc and Thibon \cite{LT}, which was proved
in level one in \cite{VV}.

There is actually a whole family of Schur-Weyl dualities
relating various different type A parabolic category $\mathcal O$'s to the 
Hecke algebra $H_d^\La$, one for each permutation of the
sequence $m_1,\dots,m_l$ (corresponding to non-conjugate parabolic
subalgebras with conjugate Levi factors). 
We usually consider only the
standard one in which $m_1 \geq \cdots \geq m_l$, as this 
leads to the most familiar combinatorics.
In the other cases the parabolic category $\mathcal O$ that appears
is usually not equivalent to the standard one, but it 
is always derived equivalent.
This leads to some interesting twisted versions of the theory.
We have included in the article some discussion of the situation for
the opposite parabolic in which $m_1 \leq \cdots \leq m_l$.
This is particularly interesting as it explains the role of tilting modules
and Ringel duality in the standard picture,
leading to what can be interpreted as 
the degenerate analogue of the results of Mathas from \cite{M}.

The rest of the article is organized as follows.
\begin{itemize}
\item
In section 2 we 
give a detailed account of the 
construction 
of the standard monomial, canonical and dual-canonical bases 
of $V(\La)$ and $F(\La)$. We also introduce a fourth basis for $F(\La)$
which we call the
{\em quasi-canonical basis}, which has similar properties to the canonical 
basis but is not invariant under the bar involution.
In the categorification picture the quasi-canonical basis 
corresponds to
indecomposable projectives, whereas the canonical basis corresponds to
tilting modules.
\item
In section 3 we derive the main categorification theorem (Theorem~\ref{cat2}) 
which relates
the three bases for $V(\La)$ just mentioned
to Specht modules, projective indecomposable
modules and
irreducible modules, respectively.
The proof of this is a straightforward application of 
Schur-Weyl duality for higher levels,
the starting point being the analogous (known)
categorification theorem relating $F(\La)$
to parabolic category $\mathcal O$, which is a consequence of the
Kazhdan-Lusztig conjecture.
\item
In section 4 we introduce Young modules and signed Young modules,
and discuss their relationship with tilting modules and Ringel duality.
\end{itemize}

\section{Combinatorics of canonical bases}

In this section, we
set up some basic notation
related to the general linear Lie algebra and its quantized
enveloping algebra $U$, in both finite and infinite rank.
Then we fix a dominant weight $\La$ and
recall a standard construction of
the irreducible $U$-module 
$V(\La)$ of highest weight $\La$, 
the main point being to explain in detail how various
natural bases fit in with this construction.

\subsection{Some combinatorics}\label{sscomb}
Let $I$ be a non-empty set of consecutive integers
and set $I_+ := I \cup (I+1)$.
Let $P := \bigoplus_{i \in I_+} \Z\La_i$ be the weight lattice associated to
the general linear Lie algebra $\mathfrak{gl}_{I_+}(\C)$
of $I_+ \times I_+$ matrices; we refer to $\La_i$ here as the
{\em $i$th fundamental weight}.
If $I$ is bounded-below then $P$ can also be written as
$P = \bigoplus_{i \in I_+} \Z \delta_i$ where the weights $\delta_i$ and $\La_j$
are related by the formula
$$
\La_j = \sum_{j \geq i \in  I_+} \delta_i.
$$
We give meaning to this formula when $I$ is not bounded-below by
embedding $P$ into a larger space $\widehat{P}$ consisting
of all formal $\Z$-linear combinations $\sum_{i \in I_+} a_i \delta_i$ such
that $a_i = 0$ for all sufficiently large $i$.
Let $Q \subset P$ be the root lattice generated by the simple
roots $\{\alpha_i\:|\:i \in I\}$ which are defined from
$$
\alpha_i = \delta_i - \delta_{i+1}.
$$
There is a canonical pairing
$(.,.):P \times Q \rightarrow \Z$
with $(\La_i, \alpha_j) = \delta_{i,j}$ for $i \in I_+, j \in I$.
If $I$ is bounded-below this pairing is the restriction of the
symmetric bilinear form on $P$ with respect to which the
$\delta_i$'s form an orthonormal basis.

Set $Q_+ := \sum_{i \in I} \Z_{\geq 0} \alpha_i$ and $P_+ := \sum_{i \in I_+} \Z_{\geq 0} \La_i$.
Let $\height(\alpha)$ denote the usual {\em height}
of $\alpha \in Q_+$, i.e. the sum of its coefficients 
when expressed as a linear combination of simple roots.
We will often need to work with $d$-tuples $\bi = (i_1,\dots,i_d) \in I^d$.
The symmetric group $S_d$ acts naturally on $I^d$ by place
permutation, and the orbits of $S_d$ on $I^d$ are the sets
$$
I^\alpha := \left\{\bi \in I^d\:|\:\alpha_{i_1}+\cdots+\alpha_{i_d} = \alpha\right\}
$$
parametrized by all $\alpha \in Q_+$ of height $d$.

Suppose we are given a dominant weight $\La \in P_+$.
We can write it uniquely as
\begin{equation}\label{ms}
\La = \La_{m_1} + \cdots + \La_{m_l}
\end{equation}
for some $l \geq 0$ and 
$m_1 \geq \cdots \geq m_l$.
We refer to $l$ here as the {\em level}.
We often identify $\La$ with its {\em diagram},
namely, the array of boxes with rows indexed by $I_+$ 
in increasing order from bottom to top,
columns indexed by $1,\dots,l$ from left to right, and a box
in row $i$ and column $j$ whenever $i \leq m_j$. This definition
makes sense even if the index set is not bounded-below, but in that case
the diagram goes down forever.
For example, taking $I = \Z$,
the diagram 
of $\La = \La_2+\La_2 + \La_1 + \La_{-1}$ is
$$
\begin{picture}(0,80)
\put(-10,44)
{
\diagram{
$ 2$&$ 2$\cr
$ 1$&$ 1$&$ 1$\cr
$ 0$&$ 0$&$ 0$\cr
$ \text{\!\! -1}$&$ \text{\!\! -1}$&$ \text{\!\! -1}$&$ \text{\!\! -1}$\cr
$ \text{\!\! -2}$&$ \text{\!\! -2}$&$ \text{\!\! -2}$&$ \text{\!\! -2}$\cr
}}
\put(16,1){$\vdots$}
\end{picture}
$$
Here we have labelled all boxes on
the $i$th row by $i$.

A {\em $\La$-tableau} is a diagram $A$ obtained by writing integers into the boxes of the diagram $\La$.
If $A$ is any $\La$-tableau, we let $A(i,j)$ denote the entry
in the $i$th row and $j$th column of $A$.
We say $A$ is {\em column-strict} if its
entries belong to $I_+$, they are strictly increasing from bottom to top in each column, and moreover
all entries in the $i$th row are equal to $i$
for all but finitely many rows (the final condition being vacuous if $I$ 
is bounded-below). 
We say $A$ is {\em standard} if it is column-strict and in addition its
entries are weakly increasing from left to right in each row.
Let $\Col^\La$ (resp.\ $\Std^\La$) denote the set
of all column-strict (resp.\ standard) $\La$-tableaux.

For any column-strict $\La$-tableau $A$ and $1 \leq j \leq l$, 
define
\begin{align}
\wt_j(A) &:=
\sum_{m_j \geq i \in I_+} \delta_{A(i,j)} 
= \La_{m_j} - \sum_{m_j \geq i \in I} 
(\alpha_{i}+\alpha_{i+1}+\cdots+\alpha_{A(i,j)-1}),\\
\wt(A) &:= 
\wt_1(A) + \cdots + \wt_l(A),\label{wtdef}
\end{align} 
the {\em weight of the $j$th column} of $A$, and the {\em weight}
of $A$, respectively.
For any $\alpha \in Q_+$, let $\Col^\La_\alpha$
(resp.\ $\Std^\La_\alpha$) denote the set of all
column-strict (resp.\ standard) $\La$-tableaux of weight
$\La - \alpha$.
There is a unique column-strict $\La$-tableau of weight $\La$,
namely, the {\em ground-state tableau $A^\La$} which has all entries
in its $i$th row equal to $i$ for all $i \in I_+$.

We also need the 
{\em Bruhat order} on $\Col^\La$.
This is defined by
$A \leq B$ if $\wt(A) = \wt(B)$
and
$$
\wt_1(A)+\cdots+\wt_j(A) \geq \wt_1(B)+\cdots+\wt_j(B)
$$
in the dominance ordering on $P$
for each $j=1,\dots,l-1$.
The Bruhat order 
has the basic property that $A < B$ if $B$ is obtained from $A$ by swapping
entries $a < b$ in columns $i < j$.

\subsection{\boldmath The standard monomial basis}\label{ssmb}
Let $U$ be the generic quantized enveloping algebra associated to 
$\mathfrak{gl}_{I_+}(\C)$.
Thus $U$ is the associative algebra over the field of rational functions $\Q(q)$ in an indeterminate $q$, with generators 
$$
\{D_i, D_i^{-1}\:\big|\:i \in I_+\}\cup\left\{E_i, F_i\:|\:i \in I\right\}
$$
subject to the following well known relations:
\begin{align*}
D_iD_i^{-1} &= D_i^{-1}D_i = 1,&E_iE_j &= E_jE_i&\hbox{if $|i-j|>1$},\\
D_iD_j &= D_jD_i,&\!\!\!\!E_i^2E_j  + E_jE_i^2 &= (q+q^{-1})E_iE_jE_i&\hbox{if $|i-j|=1$},\\
D_iE_jD_i^{-1} &= q^{(\delta_i, \alpha_j)} E_j,&F_iF_j &= F_jF_i&\hbox{if $|i-j| > 1$,}\\
D_iF_jD_i^{-1} &= q^{-(\delta_i, \alpha_j)} F_j,&\!\!\!\!F_i^2F_j+ F_jF_i^2 &=   (q+q^{-1})F_iF_jF_i&\hbox{if $|i-j|=1$,}
\end{align*}
\begin{equation*}
\quad 
E_iF_j - F_jE_i = \delta_{i, j} \frac{D_i D_{i+1}^{-1} - D_{i+1} D_i^{-1}}{q - q^{-1}}.
\end{equation*}
We view $U$ as a Hopf algebra with comultiplication
$\Delta$ defined on generators by
\begin{align*}
\Delta(D_i^{\pm 1}) &= D_i^{\pm 1} \otimes D_i^{\pm 1},\\
\Delta(E_i) &= 1 \otimes E_i + E_i \otimes D_i D_{i+1}^{-1},\\
\Delta(F_i) &= F_i \otimes 1 + D_i^{-1}D_{i+1}  \otimes F_i.
\end{align*}

For $\La \in P_+$, let
$V(\La)$ denote the irreducible $U$-module 
of highest weight $\La$, that is, the (unique up to isomorphism)
irreducible $U$-module generated by a vector $v_\La$
such that $E_i v_\La = 0$ for each $i \in I$
and $D_i v_\La = q^{(\La,\delta_i)} v_\La$
for each $i \in I_+$.
We are going to recall a well known direct construction of
$V(\La)$, beginning with the easiest case when $\La$
is a fundamental weight.

Let $V$ denote the natural $U$-module
with basis
$\{v_i\:|\:i \in I_+\}$. The generators act on this basis by
the rules
$$
D_i v_j = q^{\delta_{i,j}} v_j,
\qquad
E_i v_j = \delta_{i+1, j} v_i,
\qquad
F_i v_j = \delta_{i,j} v_{i+1}.
$$
Following \cite[$\S$5]{dual} (noting the roles of $q$
and $q^{-1}$ are switched there),
we define the $n$th {\em quantum exterior power} $\bigwedge^n V$
to be the $U$-submodule of $\bigotimes^n V$
spanned
by the vectors
\begin{align}\label{bvs}
v_{i_1} \wedge\cdots \wedge v_{i_n} &:=
\sum_{w \in S_n} (-q)^{\ell(w)} v_{i_{w(1)}} \otimes\cdots\otimes v_{i_{w(n)}}
\end{align}
for all $i_1 > \cdots > i_n$ from the index set $I_+$.
Here, $\ell(w)$ denotes the usual length of a permutation
$w\in S_n$.

If $I$ is bounded-below we have simply that
\begin{equation}\label{siw}
V(\La_m) = {\textstyle \bigwedge^{m +1- \inf(I)}} V
\end{equation}
for each $m \in I_+$.
The same thing is true if $\inf(I)=-\infty$ providing the right
hand side of (\ref{siw}) is interpreted as the {\em semi-infinite wedge}
$\bigwedge^{m+\infty} V$, that is, the $U$-module with
basis consisting of all expressions of the form
$$
v_{i_1} \wedge v_{i_2} \wedge\cdots
$$
for $I_+ \ni i_1 > i_2 > \cdots$
such that $i_n = m+1-n$ for $n \gg 0$.

More formally, $\bigwedge^{m+\infty} V$ is a direct limit of 
finite exterior powers. 
To write this down precisely, we need to let the index set $I$ vary:
for any $k \in I$ let $I_{\geq k} := \{i \in I\:|\:i \geq k\}$
and write $U_{\geq k}$ and $V_{\geq k}$
for the analogues of $U$ and $V$ defined with respect to the
truncated index set $I_{\geq k}$. Obviously $U$ and $V$
are the unions of the corresponding truncated objects taken over all $k \in I$.
Moreover for $k \leq m$ with $k-1\in I$ there is a natural embedding of 
$U_{\geq k}$-modules
\begin{equation}\label{iota}
\iota_{k}:{\textstyle\bigwedge^{m+1-k}} V_{\geq k}
\hookrightarrow
{\textstyle\bigwedge^{m+2-k}} V_{\geq (k-1)}
\end{equation}
sending 
$v_{i_1} \wedge \cdots \wedge v_{i_{m+1-k}}$
to 
$v_{i_1} \wedge \cdots \wedge v_{i_{m+1-k}} \wedge v_{k-1}$.
When $I$ is not bounded-below, the semi-infinite wedge 
$\bigwedge^{m+\infty} V$
is the direct limit of 
the finite exterior powers $\bigwedge^{m+1-k} V_{\geq k}$
over all $k \leq m$, taken with respect to the embeddings $\iota_{k}$.
It is a $U_{\geq k}$-module for each $k$, hence actually a $U$-module.

Now suppose that $\La\in P_+$
is a dominant weight of arbitrary level and write it in the standard form
(\ref{ms}). 
Set
\begin{equation}\label{bw}
F(\La)
:= 
V(\La_{m_1}) \otimes\cdots\otimes V(\La_{m_l}),
\end{equation}
a tensor product of $l$ fundamental representations.
This $U$-module has an obvious monomial basis parametrized by the
set $\Col^\La$ of column-strict $\La$-tableaux.
More precisely, given $A \in \Col^\La$, we set
\begin{equation*}
M_A := 
(v_{A(m_1,1)} \wedge v_{A(m_1-1,1)} \wedge \cdots)
\otimes\cdots\otimes
(v_{A(m_l,l)} \wedge v_{A(m_l-1,l)} \wedge \cdots),
\end{equation*}
which is a tensor product of finite or semi-infinite wedges 
according to whether $I$ is bounded-below or not;
informally, $M_A$ is the monomial in $F(\La)$ obtained by reading
the entries of $A$ down columns starting from the leftmost column.
Then the {\em monomial basis} of $F(\La)$ is the set
\begin{equation}
\left\{M_A\:\big|\:A \in \Col^\La\right\}.
\end{equation}
Each vector $M_A$ in this basis is of weight $\wt(A)$ as in (\ref{wtdef}),
justifying that notation.

Since the $\La$-weight space of $F(\La)$ is one dimensional
and all other weights are strictly smaller in the dominance order,
$V(\La)$ appears as a constituent of the integrable module
$F(\La)$ with multiplicity one.
More precisely, the vector $v_\La := M_{A^\La}$ is a canonical highest weight vector
in $F(\La)$ of weight $\La$ and we can {\em define}
$V(\La)$ to be the $U$-submodule of $F(\La)$
generated by this vector. 
In view of complete reducibility,
there is also a canonical $U$-equivariant projection
\begin{equation}\label{pie}
\pi:F(\La) \twoheadrightarrow V(\La).
\end{equation}
For each $A \in \Col^\La$, let 
$S_A := \pi(M_A) \in V(\La)$. These are the
{\em standard monomials} in $V(\La)$, and the 
{\em standard monomial basis theorem} asserts that
vectors
\begin{equation}\label{smb}
\left\{S_A\:\big|\:A \in \Std^\La\right\}
\end{equation}
give a basis for $V(\La)$. This is the quantum analogue
of the classical standard monomial basis for the space of
global sections of a line bundle on the flag variety, which goes back at least to
Hodge.
If $I$ is bounded-below the standard monomial 
basis theorem is proved (by no means for the first time!) in
\cite[Theorem 26]{dual}.
The standard monomial basis theorem when $I$ is not bounded-below follows 
easily 
from the bounded-below case by taking direct limits, as we explain in the 
next paragraph.

Suppose then that $I$ is not bounded-below and take any $k \leq m_l$.
Write $F(\La)_{\geq k}$ and 
$V(\La)_{\geq k}$ for the analogues of the modules 
$F(\La)$ and
$V(\La)$ over the truncated
algebra 
$U_{\geq k}$, i.e. working with the bounded-below index set $I_{\geq k}$. 
One checks easily that there is a commutative diagram
\begin{equation}\label{uiota}
\begin{CD}
F(\La)_{\geq k} &@>\pi>>& V(\La)_{\geq k}\\
@V\iota_k VV&&@VV\iota_k V\\
F(\La)_{\geq (k-1)} &@>>\pi>& V(\La)_{\geq (k-1)}
\end{CD}
\end{equation}
of $U_{\geq k}$-module homomorphisms.
Here, the left hand vertical map is the tensor product of
$l$ maps of the form (\ref{iota}), and we have that
\begin{equation}\label{mcup}
\iota_k(M_A) = M_{A\sqcup(k-1)^l},
\end{equation}
where $A \sqcup (k-1)^l$
denotes the tableau obtained
 obtained by adding an extra row of
$l$ boxes, each labelled by $(k-1)$, to the bottom of $A$.
The right hand vertical map is defined as the unique $U_{\geq k}$-module homomorphism
mapping $v_\La \in V(\La)_{\geq k}$ to $v_\La \in V(\La)_{\geq (k-1)}$.
The commutativity of the diagram implies that the right hand
$\iota_k$ also has the property that
\begin{equation}\label{smcup}
\iota_k(S_A) = S_{A \sqcup (k-1)^l}.
\end{equation}
Moreover $F(\La)$ and $V(\La)$ are the direct limits of
their truncated versions over all $k \leq m_l$,
again taken with respect to the embeddings $\iota_k$.
In view of (\ref{mcup}), the monomial basis for
$F(\La)$ is the union of the monomial bases from
all bounded-below cases. Similarly by (\ref{smcup}),
the standard monomial basis for $V(\La)$ 
is the union of the standard monomial bases from all 
$V(\La)_{\geq k}$. This proves the standard monomial basis theorem
for $V(\La)$ in the case that $I$ is not bounded-below.

\subsection{The dual-canonical basis}\label{ssdcb}
In this subsection we are going to recall 
a slightly unorthodox definition of 
Lusztig's dual-canonical 
basis of $V(\La)$ (which is the upper global crystal base of Kashiwara)
following \cite[$\S$7]{dual}.
To get started, we need Lusztig's bar involution on $F(\La)$.
The {\em bar involution} on $U$ is the
automorphism $-:U \rightarrow U$ that is
anti-linear with respect to the field automorphism
$\Q(q) \rightarrow \Q(q), f(q) \mapsto f(q^{-1})$
and satisfies
\begin{equation*}
\overline{E_i} = E_i,\qquad
\overline{F_i} = F_i,\qquad
\overline{D_i} = D_i^{-1}.
\end{equation*}
By a {\em compatible bar involution} on a $U$-module $M$
we mean an anti-linear involution
$-:M \rightarrow M$
such that $\overline{uv} = \overline{u}\,\overline{v}$
for each $u \in U, v \in M$.
The fundamental module $V(\La_m)$ possesses a compatible bar involution
which fixes each of its basis vectors
of the form $v_{i_1} \wedge v_{i_2} \wedge \cdots$
for $i_1 > i_2 > \cdots$.
From this and a general construction due to Lusztig
involving the quasi-$R$-matrix \cite[$\S$27.3]{Lubook}, we get 
a compatible bar involution on the tensor product
$F(\La)$ from (\ref{bw}).
It has the crucial property that
\begin{equation*}
\overline{M_A} = M_A + \text{(a $\Z[q,q^{-1}]$-linear combination of 
$M_B$'s for $B < A$)}
\end{equation*} 
where $<$ is the Bruhat order on column-strict tableaux.
This is explained in more detail in 
\cite[$\S$5]{dual} in the case when $I$ is bounded-below.
In view of the following lemma, the bar involution on
$F(\La)$ when $I$ is not bounded-below is the 
limit of the bar involutions on each of the truncations
$F(\La)_{\geq k}$.

\begin{Lemma}\label{stable}
Suppose that $I$ is not bounded-below and take $k \leq m_l$.
The bar involution commutes with the natural embedding 
$\iota_k:F(\La)_{\geq k}
\hookrightarrow F(\La)_{\geq (k-1)}$.
\end{Lemma}

\begin{proof}
This follows from the definition of the bar involution and the fact that
the quasi-$R$-matrix attached to
$U_{\geq (k-1)}$ is equal to the quasi-$R$-matrix attached to
$U_{\geq k}$ plus a sum of terms which annihilate 
vectors in $\iota_k(F(\La)_{\geq k})$ by weight considerations.
\end{proof}

Applying 
Lusztig's lemma \cite[Lemma 24.2.1]{Lubook},
we can now define the {\em dual-canonical basis}
\begin{equation}
\left\{L_A\:\big|\:A \in \Col^\La\right\}
\end{equation}
of $F(\La)$ by declaring that $L_A$ is the unique bar-invariant vector 
such that
$$
L_A = M_A + \text{(a $q \Z[q]$-linear combination of
$M_B$'s for various $B<A$)}.
$$
Using Lemma~\ref{stable}, (\ref{mcup})
and the definition, it is easy to check 
for any $k \leq m_l$ with $k-1\in I$
that
the embedding
 $\iota_k:F(\La)_{\geq k}
\hookrightarrow F(\La)_{\geq (k-1)}$
satisfies
\begin{equation}\label{lstab}
\iota_k(L_A) = L_{A \sqcup (k-1)^l}.
\end{equation}
This means that, like the monomial basis, 
the dual-canonical basis for $F(\La)$
in the case that $I$ is not bounded-below
is the union of the dual-canonical bases from all bounded-below cases.

The polynomials $d_{A,B}(q), p_{A,B}(q) \in \Z[q]$ defined from
\begin{align}\label{t}
M_B &= \sum_{A \in \Col^\La} d_{A,B}(q) L_A,\\
L_B &= \sum_{A \in \Col^\La} p_{A,B}(-q) M_A\label{e2}
\end{align}
satisfy $p_{A,A}(q) = d_{A,A}(q) = 1$ and $p_{A,B}(q) = d_{A,B}(q) = 0$
unless $A \leq B$.
In \cite[Remark 14]{dual} one finds explicit formulae expressing these polynomials in terms of finite
type A Kazhdan-Lusztig polynomials. Let us record the appropriate formla for the $p_{A,B}(q)$'s, which
are just Deodhar's parabolic Kazhdan-Lusztig polynomials for the symmetric group.
It suffices in view of (\ref{mcup}) and (\ref{lstab}) 
to do this in the case that $I$ is bounded-below.
Consider the natural right action of
$S_{n}$ on $I_+^n$ by place permutation.
For $A \in \Col^\La$, let 
$(a_1,\dots,a_n)$ be 
{\em column-reading} of $A$, that is, 
the tuple obtained by reading the entries
of $A$ down columns starting with the leftmost column.
Define $w_A$ to be the unique element of
$S_{n}$ of minimal length such that
$(a_1,\dots,a_n)w_A$ is a weakly increasing sequence.
Also let $Z_A$ be the stabilizer in $S_{n}$ of this
weakly increasing sequence.
Then for any $A \leq B$ we have that
\begin{equation}\label{scf}
p_{A,B}(q) = q^{\ell(w_B) - \ell(w_A)}
\sum_{z \in Z_A} (-1)^{\ell(z)} P_{w_A z, w_B}(q^{-2}),
\end{equation}
where $P_{x,y}(q)$ is the usual Kazhdan-Lusztig polynomial
exactly as in \cite{KL1}.
The formula (\ref{scf}) should be compared with 
the second formula from \cite[Theorem 3.11.4(iv)]{BGS}.

Now we pass from $F(\La)$ to $V(\La)$.
The bar involution on $F(\La)$ restricts to give a 
well-defined compatible bar involution on $V(\La)$.
It is the unique compatible bar involution on $V(\La)$
fixing the highest weight vector $v_\La$.
For $A\in \Std^\La$ define
$D_A := \pi(L_A) \in V(\La)$, so that
\begin{equation}
D_B = \sum_{A \in \Col^\La} p_{A,B}(-q) S_A.
\end{equation}
In view of the following theorem, the vectors
\begin{equation}
\left\{D_A\:\big|\:A \in \Std^\La\right\}
\end{equation}
constitute the {\em dual-canonical basis} of $V(\La)$.
Moreover we have that
\begin{equation}\label{sind}
S_B = \sum_{A \in \Std^\La} d_{A,B}(q) D_A
\end{equation}
for $B \in \Col^\La$.
Finally by (\ref{lstab}) and the commutativity of the diagram (\ref{uiota}),
we again get that
\begin{equation}\label{dstab}
\iota_k(D_A) = D_{A \sqcup (k-1)^l},
\end{equation}
i.e. the dual-canonical basis for $V(\La)$ when $I$ is not bounded-below
is the union of the dual-canonical bases from all the bounded-below cases.

\begin{Theorem}\label{las}
Each $D_A$ is bar-invariant
and the vectors $\{D_A\:|\:A \in \Std^\La\}$ form a basis for
$V(\La)$ which coincides with Lusztig's dual-canonical basis.
Moreover, 
$\pi(L_A) = 0$
if  $A$ is not standard.
\end{Theorem}

\begin{proof}
It suffices to prove the theorem in the case that $I$ is bounded-below.
In that case, everything follows from \cite[Theorem 26]{dual}.
The fact that the $D_A$'s coincide with the usual dual-canonical basis in the
sense of Lusztig is explained in \cite[Remark 27]{dual}.
\end{proof}

\begin{Corollary}\label{kerb}
The vectors $\{L_A\:|\:A \in \Col^\La \setminus \Std^\La\}$
give a basis for $\ker \pi$.
\end{Corollary}

\subsection{The crystal graph} \label{sscg}
The actions of the Chevalley generators
$E_i$ and $F_i$ on the dual-canonical bases 
of $F(\La)$ and $V(\La)$ 
are reflected by an underlying {\em crystal graph}. 
We wish to briefly recall the definition of this
important combinatorial object.

Suppose we are given a column-strict tableau $A \in \Col^\La$
and $i \in I$.
Enumerate the boxes of $A$
that contain the entries $i$ or $(i+1)$
as $b_1,\dots,b_n$ in column-reading order, i.e.
working in order down columns starting with the
leftmost column. Thus if $r < s$ then $b_r$ is either located in a column
strictly
to the left of $b_s$, or $b_r$ is in the same column
as but strictly above $b_s$.
Then we define the
{\em reduced $i$-signature} $(\sigma_1,\dots,\sigma_n)$ 
of $A$ by applying the following algorithm.
Start with the sequence $(\sigma_1,\dots,\sigma_n)$ in which
$\sigma_r=+$ if $b_r$ contains the entry $i$
and $\sigma_r=-$ if $b_r$ contains the entry $(i+1)$.
If we can find $1 \leq r < s \leq n$
such that $\sigma_r = -$, $\sigma_s = +$
and $\sigma_{r+1} =\cdots=\sigma_{s-1}=0$
then we replace $\sigma_r$ and $\sigma_s$ by $0$.
Keep doing this until we are left with a sequence
$(\sigma_1,\dots,\sigma_n)$
in which no $-$ appears to the left of a $+$.
This is the {reduced $i$-signature} of $A$.
Then define
\begin{align}
\eps_i(A) &:= \#\{r=1,\dots,n\:|\:\sigma_r=-\},\label{epsi}\\
\phi_i(A) &:= \#\{r=1,\dots,n\:|\:\sigma_r=+\}.
\end{align}
If $\eps_i(A) = 0$ then we set
$\tilde e_i(A) := \frownie$; otherwise we define
$\tilde e_i(A)$ to be the column-strict tableau obtained
by replacing the entry $(i+1)$ in box $b_r$ by $i$,
where $r$ indexes the leftmost $-$ in the reduced $i$-signature.
Similarly, if $\phi_i(A) = 0$ then we set
$\tilde f_i(A) := \frownie$; otherwise we define
$\tilde f_i(A)$ to be the column-strict tableau obtained
by replacing the $i$ in box $b_r$ by $(i+1)$,
where $r$ indexes the rightmost $+$ in the reduced $i$-signature.
This defines the crystal operators
\begin{equation}\label{te}
\tilde e_i, \tilde f_i: \Col^\La \rightarrow \Col^\La \sqcup \{\frownie\}.
\end{equation}
Moreover it is obviously the case that
\begin{equation}
(\wt(A), \alpha_i) = \phi_i(A) - \eps_i(A).
\end{equation}
The datum $(\Col^\La, \tilde e_i, \tilde f_i, 
\eps_i, \phi_i, \wt)$ just defined
is the crystal in the sense of Kashiwara
\cite{Ka}
associated to the module $F(\La)$.

The {\em crystal graph} is then the colored, directed graph with
vertex set $\Col^\La$ and an edge
$A \stackrel{i}{\rightarrow} B$ of color $i$
if $B = \tilde f_i (A)$; equivalently, 
$A = \tilde e_i (B)$. 
The connected component of this graph containing the
ground-state tableau $A^\La$ 
has vertex set $\Std^\La$.
This connected component is the crystal graph associated
to the highest weight module $V(\La)$; 
for $I$ bounded-below
it is the same
as the crystal graph from \cite{KN}.

\begin{Lemma}\label{dca}
For $A \in \Col^\La$ and $i \in I$ we have that
\begin{align*}
E_i L_A &=
[\eps_i(A)] L_{\tilde e_i(A)}
+ \sum_{\substack{B \in \Col^\La \\
\eps_i(B) < \eps_i(A)-1
}}
x_{A,B}^i(q) L_B,\\
F_i L_A &=
[\phi_i(A)] L_{\tilde f_i(A)}
+ \sum_{\substack{B \in \Col^\La \\
\phi_i(B) < \phi_i(A)-1
}}
y_{A,B}^i(q) L_B,
\end{align*}
for bar-invariant
 $x_{A,B}^i(q) \in q^{\eps_i(A)-2} \Z[q^{-1}]$ and
$y_{A,B}^i(q) \in q^{\phi_i(A)-2} \Z[q^{-1}]$.
The analogous statement 
with $L_A$ replaced by $D_A$
and $\Col^\La$ replaced by $\Std^\La$ everywhere
is also true.
\end{Lemma}

\begin{proof}
Our dual-canonical basis is the upper global crystal base
in the sense of Kashiwara associated to the tensor product 
(\ref{bw}), and the crystal structure defined above is precisely 
the underlying 
crystal by Kashiwara's tensor product rule for crystals.
Given this, the first part of the lemma follows from
\cite[Proposition 5.3.1]{KaG}.
The final statement then follows on applying the 
projection $\pi$, using the final statement of Theorem~\ref{las}.
\end{proof}

\subsection{A twisted version}\label{stv}
We can modify the construction of $V(\La)$
and its dual-canonical basis
by changing the order of the tensor product
of fundamental representations in the definition (\ref{bw})
of the module $F(\La)$. As explained in detail in \cite[Theorem 26]{dual},
this modification leads naturally
to a family of monomial bases for $V(\La)$,
one for each permutation of the sequence
$m_1,\dots,m_l$, but always produces the same dual-canonical basis at the end. 
We want to briefly explain one of these twisted versions,
namely,
the one which is at the opposite extreme to the construction
explained so far.

Continue with $\La$ fixed as in (\ref{ms}),
so that $m_1 \geq \cdots \geq m_l$.
Define
\begin{equation}
\widetilde{F}(\La) := V(\La_{m_l}) \otimes\cdots\otimes V(\La_{m_1}),
\end{equation}
so we have taken the tensor product in the reverse order to 
that of $\S$\ref{ssmb}.
Let
$$
\widetilde{M}_A := (v_{A(m_l,l)} \wedge v_{A(m_l -1,l)}\wedge\cdots) \otimes \cdots \otimes
(v_{A(m_1,1)} \wedge v_{A(m_1-1,1)} \wedge \cdots)
$$
denote the monomial obtained by
reading the entries of $A$ down columns starting from
the rightmost column;
we refer to this way of reading the entries of $A$ as
{\em reverse-column-reading}.
This gives us the obvious monomial basis for this space:
\begin{equation}\label{b1}
\left\{\widetilde{M}_A\:\big|\:A \in \Col^\La\right\}.
\end{equation} 
There is also a bar involution on $\widetilde{F}(\La)$ defined exactly 
as before. From this, we get
the dual-canonical basis
\begin{equation}\label{b2}
\left\{\widetilde{L}_A\:\big|\:A \in \Col^\La\right\}
\end{equation}
in which $\widetilde{L}_A$ is the unique bar-invariant vector such that
$$
\widetilde{L}_A = \widetilde{M}_A + \text{(a $q \Z[q]$-linear combination of
$\widetilde{M}_B$'s for various $B>A$)}.
$$
Note the inequality $B > A$ here is the reverse of the
analogous inequality in the definition of $L_A$ in $\S$\ref{ssdcb}.
We get a twisted version of the polynomials
from (\ref{t}) by expanding
\begin{align}\label{t2}
\widetilde{M}_B &= \sum_{A \in \Col^\La} \widetilde{d}_{A,B}(q) \widetilde{L}_A.
\end{align}
Thus $\widetilde{d}_{A,A}(q) = 1$ and $\widetilde{d}_{A,B}(q) = 0$
unless $A \geq B$.

The crystal graph in this setting is defined in a similar way to
the crystal graph in $\S$\ref{sscg}, but starting 
from the enumeration $b_1,\dots,b_n$ of the boxes of
$A$ containing the entries $i$ or $(i+1)$
in reverse-column-reading order, i.e. if $r < s$ then $b_r$ 
is either to the right of $b_s$ or it is in the same column but
strictly above $b_s$. We call the resulting crystal structure
on the set $\Col^\La$
the {\em reverse crystal structure}.
Let $\RStd^\La$ 
denote the subset of $\Col^\La$ that indexes the vertices
from the connected component of the reverse crystal graph
generated by the ground-state tableau $A^\La$,
and set $\RStd^\La_\alpha := \RStd^\La \cap \Col^\La_\alpha.$
The set $\RStd^\La$ can be described directly as the set of all
{\em reverse-standard} $\La$-tableaux, that is,
the column-strict $\La$-tableaux with the property that,
on sliding all boxes in the $i$th column up by $(m_1-m_i)$ places,
the entries within each row are weakly decreasing from left to right;
this combinatorial description can be
derived from \cite[(2.2)]{dual}.

There is a canonical crystal isomorphism
\begin{equation}\label{rect}
\RStd^\La \stackrel{\sim}{\rightarrow} \Std^\La,
\quad
A \mapsto A^{\downarrow}
\end{equation}
between $\RStd^\La$ equipped with the reverse crystal structure
and $\Std^\La$ equipped with the usual crystal structure.
This map can easily be computed as a special case of the
{\em rectification map} of 
Lascoux and Sch\"utzenberger \cite{LS}; see
also \cite[$\S$A.5]{Fulton}
and \cite[(2.3)]{dual}.
We recall this explicitly after the statement of Theorem~\ref{thet}
below.

Again we have a canonical
projection
\begin{equation}
\widetilde{\pi}:\widetilde{F}(\La) \twoheadrightarrow V(\La)
\end{equation}
mapping $\widetilde{M}_{A^\La}$ to $v_\La$. 
We set $\widetilde{S}_A := \widetilde{\pi}(\widetilde{M}_A)$
for each $A \in \Col^\La$ and 
$\widetilde{D}_A := \widetilde{\pi}(\widetilde{L}_A)$
for $A \in \RStd^\La$.
For $I$ bounded-below the following theorem is 
another special case of \cite[Theorem 26]{dual}; it extends to arbitrary
$I$ by the usual direct limit argument.

\begin{Theorem}\label{thet}
The vectors $\{\widetilde{S}_A\:|\:A \in \RStd^\La\}$ and
$\{\widetilde{D}_A\:|\:A \in \RStd^\La\}$ give bases
for $V(\La)$. Moreover for each $A \in \RStd^\La$ we have that
$$
\widetilde{D}_A = D_{A^{\downarrow}}.
$$
Hence the basis $\{\widetilde{D}_A\:|\:A \in \RStd^\La\}$
coincides with Lusztig's dual-canonical basis (but it is parametrized in a non-standard way).
\end{Theorem}

The rectification map (\ref{rect}) can be computed explicitly as follows.

Assume to start with that $I$ is
bounded-below, and recall the notion of
{\em row insertion} from \cite[$\S$1.1]{Fulton}.
Given $A \in \RStd^\La$, 
read the entries of $A$ in reverse-column-reading order
to obtain a sequence of integers
$a_1,\dots,a_n$.
Then, starting from the diagram of $\La$ with all boxes empty,
we use row insertion to 
successively insert the entries $a_1,\dots,a_n$ into the 
bottom row, bumping larger entries up and writing entries within each
row in weakly increasing order,
to obtain $A^\downarrow$ at the end.
This means that 
at the $r$th step the diagram has $(r-1)$ boxes filled in
and we need to insert the entry $a_r$ into the bottom row.
If $a_r$ is $\geq$ all entries in this row, simply add it to the 
first empty box in the row;
else find the smallest entry $b$ in the row that is strictly larger than
$a_r$, replace this entry $b$ by $a_r$, then insert $b$ into the next row up in a similar fashion.

To drop the assumption that $I$ is bounded-below,
we just note
that
$$
(A \sqcup (k-1)^l)^{\downarrow} =A^{\downarrow} \sqcup (k-1)^l
$$
if all entries of $A$ are $\geq k$.
Hence $A^\downarrow$
when $I$ is not bounded-below can be computed
by first choosing $k$ so that all entries in the $(k-1)$th
row of $A$ equal $k-1$,
then applying the above algorithm to 
the bounded-below tableau in
rows $k$ and above, leaving rows below the $k$th row 
untouched.

Here are some examples
of the map $A \mapsto A^\downarrow$:
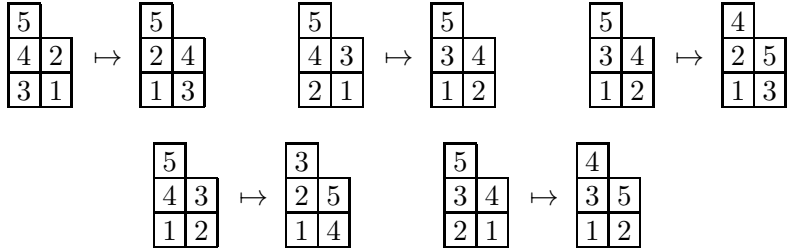
$$
\begin{picture}(0,100)
\put(-150,72){\diagram{5\cr 4&2 \cr 3&1 \cr}}
\put(-100,72){\diagram{5\cr 2&4 \cr 1&3 \cr}}
\put(-118,71){$\mapsto$}

\put(-40,72){\diagram{5\cr 4&3 \cr 2&1 \cr}}
\put(10,72){\diagram{5\cr 3&4 \cr 1&2 \cr}}
\put(-8,71){$\mapsto$}

\put(70,72){\diagram{5\cr 3&4 \cr 1&2 \cr}}
\put(120,72){\diagram{4\cr 2&5 \cr 1&3 \cr}}
\put(102,71){$\mapsto$}

\put(-95,19){\diagram{5\cr 4&3 \cr 1&2 \cr}}
\put(-45,19){\diagram{3\cr 2&5 \cr 1&4 \cr}}
\put(-63,18){$\mapsto$}

\put(15,19){\diagram{5\cr 3&4 \cr 2&1 \cr}}
\put(65,19){\diagram{4\cr 3&5 \cr 1&2 \cr}}
\put(47,18){$\mapsto$}
\end{picture}
$$
As these examples may suggest, it is always the case that 
$A \geq A^\downarrow$ in the Bruhat order, as follows
from Corollary~\ref{yu} below.

The inverse of the map (\ref{rect}) 
gives another weight-preserving bijection
\begin{equation}\label{rect2}
\Std^\La \stackrel{\sim}{\rightarrow} \RStd^\La,
\quad
A \mapsto A^{\uparrow}.
\end{equation}
This can also be computed in terms of some
row insertions. We explain just in the case that $I$ is bounded-below.
Take $A \in \Std^\La$ and read its entries 
$a_1,\dots,a_n$ in column order.
Start with the empty diagram obtained 
by sliding all boxes in the $i$th column of the diagram of $\La$
up by $(m_1-m_i)$ places.
Then successively insert $a_n,\dots,a_1$ into the top row of 
this diagram,
this time bumping smaller entries down and writing
entries within each row in weakly decreasing order.
Thus, at the $r$th step, we need to insert $a_{n+1-r}$ into the top row.
If $a_{n+1-r}$ is $\leq$ all entries in this row, we simply add
it to the first empty box in the row;
else find the largest entry $b$ in the row that is strictly
smaller than $a_{n+1-r}$, replace $b$ by $a_{n+1-r}$, then
insert $b$ into the next row down in a similar fashion.
When all insertions are done, we then slide all boxes
in the $i$th column back down by $(m_1-m_i)$ places to end up with
the desired tableau $A^\uparrow$ of shape $\La$.
It is the case that $A^\uparrow \geq A$ in the Bruhat order.

\subsection{The canonical and quasi-canonical bases}
We are ready to define two more natural bases for $F(\La)$.
First, we have the {\em canonical basis}
\begin{equation}
\left\{T_A\:\big|\:A \in \Col^\La\right\}
\end{equation}
from 
\cite[$\S$27.3]{Lubook}.
By definition, $T_A$
is the unique bar-invariant vector in $F(\La)$
such that
$$
T_A = M_A + \text{(a $q^{-1}\Z[q^{-1}]$-linear combination
of $M_B$'s for various $B < A$)}.
$$
Second, we have the {\em quasi-canonical basis}
\begin{equation}
\left\{P_A\:\big|\:A \in \Col^\La\right\}
\end{equation}
which is defined from the equations
\begin{align}
P_A &= \sum_{B \in \Col^\La} d_{A,B}(q) M_B,\label{boo}\\
M_A &= \sum_{B \in \Col^\La} p_{A,B}(-q) P_B.
\end{align}
(We have simply transposed the transition matrices
from (\ref{t})--(\ref{e2}).)

The canonical and quasi-canonical bases 
have very similar properties, since both
are dual to the dual-canonical basis under certain 
pairings. 
In the case of the canonical basis, there is a twist
here since it is actually 
the dual basis to the dual-canonical basis on the space
$\widetilde{F}(\La)$
rather than on the space
$F(\La)$ itself.
In this article,
we usually prefer to work with the quasi-canonical basis
rather than the canonical basis, in part to avoid this
awkward twist but also because it is more convenient from the
point of view of the categorifications studied later on.

The sense in which the quasi-canonical basis is dual to the
dual-canonical basis is as follows. Introduce a sesquilinear form 
$\langle.,.\rangle$ on $F(\La)$ (anti-linear in the first argument, linear in the second)
such that
\begin{equation}\label{dr}
\langle M_A, \overline{M_B} \rangle = \delta_{A,B}
\end{equation}
for each $A, B \in \Col^\La$.
A straightforward computation using (\ref{boo})
and the formula obtained from (\ref{e2}) by applying the bar involution then
shows that
\begin{equation}\label{doe}
\langle P_A, L_B \rangle = \delta_{A,B}
\end{equation}
for $A, B \in \Col^\La$.
To formulate another property of the form $\langle.,.\rangle$, let $\tau:U \rightarrow U$
be the anti-linear anti-automorphism such that
\begin{equation}\label{taudef}
\tau(E_i) = q F_i D_i^{-1} D_{i+1},
\quad
\tau(F_i) = q^{-1} D_iD_{i+1}^{-1} E_i,
\quad
\tau(D_i) = D_i^{-1}.
\end{equation}
Then we have that
\begin{equation}\label{shap}
\langle u x, y \rangle = \langle x, \tau(u) y \rangle
\end{equation}
for $x, y \in F(\La)$ and $u \in U$.
This follows by a direct check using (\ref{dr}).

To make the sense in which the canonical basis is dual to the dual-canonical basis precise, recall the bases (\ref{b1})--(\ref{b2})
for $\widetilde{F}(\La)$.
Define a bilinear pairing
$(.,.):\widetilde{F}(\La) \times F(\La) \rightarrow 
\Q(q)$
by declaring that 
\begin{equation}
(\widetilde{M}_A, \overline{M_B}) = \delta_{A, B}
\end{equation}
for all $A, B \in \Col^\La$.
We refer to this pairing as the {\em contravariant form}.
By \cite[Theorem 11]{dual} we have that
\begin{equation}
(\widetilde{L}_A, T_B) = \delta_{A,B}
\end{equation}
for all $A, B \in \Col^\La$.
Note also by \cite[Lemma 3]{dual} that the form $(.,.)$ has the property that
\begin{equation}\label{contra}
(ux,y) = (x, u^* y)
\end{equation}
where $*:U \rightarrow U$ is the linear anti-automorphism
with $E_i^* = F_i, F_i^* = E_i$ and $D_i^* = D_i^{-1}$.

The sesquilinear form $\langle.,.\rangle$ restricts to 
a form on $V(\La)$ which is characterized uniquely by 
sesquilinearity, the property 
(\ref{shap}) and the fact that $\langle v_\La,v_\La\rangle = 1$.
We call this the {\em Shapovalov form} on $V(\La)$.
Similarly the contravariant form $(.,.)$ induces 
a form on $V(\La)$, namely, the unique symmetric bilinear form $(.,.)$ 
on $V(\La)$ such that (\ref{contra}) holds and $(v_\La,v_\La)=1$.
In fact the Shapovalov form and the contravariant form
on $V(\La)$ are closely related:

\begin{Lemma}\label{contra2}
For vectors $v, w \in V(\La)$ with $v$ of weight $\La-\alpha$,
we have that
\begin{itemize}
\item[(i)]
$\langle v,w \rangle = q^{\frac{1}{2}(2\La-\alpha,\alpha)}(\overline{v},w)$;
\item[(ii)]
$(v,w) = \overline{(\overline{w}, \overline{v})}$;
\item[(iii)]
$\langle v,w \rangle =
q^{(2\La-\alpha,\alpha)} \overline{\langle w,v \rangle}$.
\end{itemize}
\end{Lemma}

\begin{proof}
The first two equalities follow by induction on
$\height(\alpha)$; for the induction step 
consider $v = F_i v'$ for some $i \in I$
and $v'$ of weight $\La - (\alpha-\alpha_i)$
and use the defining properties (\ref{shap}) and
(\ref{contra}).
The third equality is a consequence of the first two.
\end{proof}

The following result explains the relationship between the
quasi-canonical and canonical bases of $F(\La)$ and
the usual canonical basis of the submodule
$V(\La)$ (which is Kashiwara's
lower global crystal base).
Recall the bijection from (\ref{rect2})
which identifies the two natural choices of
indexing set $\Std^\La$ and $\RStd^\La$ in this subject.

\begin{Theorem}\label{maincb}
The vectors $\{T_{A}\:|\:A \in \RStd^\La\}$
give a basis for $V(\La)$ which coincides with
Lusztig's canonical basis.
Moreover, for $A \in \Std^\La_\alpha$, we have that
$$
P_A = q^{\frac{1}{2}(2\La-\alpha,\alpha)}T_{A^\uparrow}.
$$
Hence the vectors 
$\{P_A\:|\:A \in \Std^\La\}$ also give a basis for $V(\La)$
which coincides with Lusztig's canonical basis up to rescaling each vector
by a suitable power of $q$.
Finally,
$$
\langle P_A, D_B \rangle = ( T_{A^\uparrow}, D_B ) = \delta_{A,B}
$$
for all $A, B \in \Std^\La$.
\end{Theorem}

\begin{proof}
The first statement is a special case of
\cite[Proposition 27.1.7]{Lubook}; see also
\cite[Theorem 29]{dual} where the fact that
\begin{equation}\label{step}
(T_{A^\uparrow}, D_B) = \delta_{A,B}
\end{equation}
is derived at the same time.

Next we claim that the vectors $\{P_A\:|\:A \in \Std^\La\}$
also give a basis for $V(\La)$. Let
$X$ denote the subspace of $F(\La)$ 
spanned by the vectors $\{P_A\:|\:A \in \Std^\La\}$.
In view of 
Corollary~\ref{kerb} and (\ref{doe}), we have that
$$
X = \{v \in F(\La)\:|\:\langle v,w \rangle = 0\text{ for all $w \in \ker \pi$}\}.
$$
Combining this with (\ref{shap}) and the fact that $\ker \pi$ is a $U$-submodule of $F(\La)$, it follows that $X$ is a $U$-submodule
of $F(\La)$ too.
Since $X$ contains the vector $v_\La$ (that is $P_{A^\La}$) we deduce
that $V(\La) \subseteq X$. Finally the weight spaces of $V(\La)$ and $X$
have the same dimensions, so we actually have that $V(\La) = X$.
This proves the claim.

For $A, B \in \Std^\La$, 
we have by (\ref{doe}) that
\begin{equation}\label{step2}
\langle P_A, D_B \rangle
= \langle P_A, L_B \rangle = \delta_{A,B}.
\end{equation}
On the other hand, assuming in addition that $A$ is of weight $\La-\alpha$,
we have by Lemma~\ref{contra2}(i) and (\ref{step}) that
$$
\langle q^{\frac{1}{2} (2\La-\alpha,\alpha)} T_{A^\uparrow}, D_B
\rangle = 
( T_{A^\uparrow}, D_B ) = \delta_{A,B}.
$$
Comparing with (\ref{step2})
this shows that $P_A = 
q^{\frac{1}{2} (2\La-\alpha,\alpha)} T_{A^\uparrow}$.
\end{proof}

\begin{Corollary}\label{yu}
Suppose that $A \in \Std^\La_\alpha$ and $B \in \Col^\La_\alpha$
for some $\alpha \in Q_+$, and set
$a := \frac{1}{2} (2\La-\alpha,\alpha)$.
Then,
$d_{A,B}(q) = 0$ unless $A \leq B \leq A^\uparrow$.
Moreover,
$d_{A,A}(q) = 1$, 
$d_{A, A^\uparrow}(q) = q^a$,
and
$d_{A,B}(q)$ belongs to
 $q \Z[q] \cap q^{a-1} \Z[q^{-1}]$
for $A < B < A^\uparrow$.
\end{Corollary}

\begin{proof}
We 
know by Theorem~\ref{maincb} that $P_A = q^a  T_{A^\uparrow}$.
By the definition of $T_{A^\uparrow}$ this shows that
$$
P_A = q^a M_{A^\uparrow} + \text{(a $q^{a-1}\Z[q^{-1}]$-linear
combination of $M_B$'s for $B < A^\uparrow$)}.
$$
Recalling (\ref{boo}),
this shows that $d_{A,A^\uparrow}(q) = q^a$,
$d_{A,B}(q)$ is zero unless
$B \leq A^\uparrow$,
and $d_{A,B}(q)$ belongs to $q^{a-1} \Z[q^{-1}]$
if $B < A^\uparrow$.
On the other hand, we already know 
from (\ref{t})
that $d_{A,A}(q) = 1$,
$d_{A,B}(q)$ is zero unless $B \geq A$,
and $d_{A,B}(q)$ belongs to $q \Z[q]$ 
for $B > A$.
\end{proof}

We mention finally that there is an analogue of Theorem~\ref{maincb}
(and its corollary)
in the twisted case. We only need to know one thing about this:
if we define the quasi-canonical basis elements
for $\widetilde{F}(\La)$ by setting
\begin{equation}
\widetilde{P}_A := \sum_{B \in \Col^\La} \widetilde{d}_{A,B}(q) \widetilde{M}_B,\label{boo3}
\end{equation}
where $\widetilde{d}_{A,B}(q)$ is the polynomial from (\ref{t2}),
then $\{\widetilde{P}_A\:|\:A \in \RStd^\La\}$
is again equal to Lusztig's canonical basis for $V(\La)$
up to rescaling by suitable powers of $q$.
In fact it happens that
\begin{equation}
\widetilde{P}_A = P_{A^{\downarrow}}
\end{equation}
for each $A \in \RStd^\La$,
which is the dual formula to the one in Theorem~\ref{thet}.

\subsection{Multipartitions}
Almost every combinatorial definition so far has an
alternative formulation using the language of
$l$-multipartitions instead of column-strict $\La$-tableaux.
This alternative language is particularly convenient in the case that
$I$ is not bounded-below, indeed, it is the language used in almost
all of the existing literature surrounding Ariki's categorification theorem.
In this subsection, we want to 
explain the dictionary between the two languages.

To start with let $\mathscr P$ denote the set of all 
{\em partitions}
$\rpar = (\rpar_1,\rpar_2,\dots)$ in the usual sense,
i.e. weakly decreasing sequences of non-negative integers.
We write $|\rpar|$ for 
$\rpar_1+\rpar_2+\cdots$.
We often identify $\rpar \in \mathscr P$ 
with its {\em Young diagram} drawn according to the usual
English convention, i.e. rows are indexed $1,2,3,\dots$ from top to bottom,
columns are indexed $1,2,3,\dots$ from left to right, and
there is a box in the $i$th row and $j$th column whenever
$j \leq \rpar_i$.
Here for example is the Young diagram of
$\rpar = (3,2)$:
$$
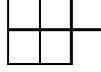
\diagram{&&\cr&\cr}
$$
We stress that the Young diagram of a partition $\rpar$
plays a quite different role in this article to the diagram of the
dominant weight $\La \in P_+$ from $\S$\ref{sscomb}.

By an {\em $l$-multipartition} we mean 
an $l$-tuple
$\bpar = (\rpar^{(1)},\dots,\rpar^{(l)}) \in \mathscr P^l$.
The {\em diagram} of an $l$-multipartition
$\bpar$ means 
the column vector containing the Young diagrams 
$\rpar^{(1)},\dots,\rpar^{(l)}$ in order from top to bottom.
If the diagram of $\bpar$ contains $d$ boxes then we 
say
that $\bpar$ is an $l$-multipartition of $d$.
Here for example is the diagram of the $2$-multipartition
$((3,2,1), (3,1))$ of 10 with boxes filled by the integers
$1,\dots,10$:
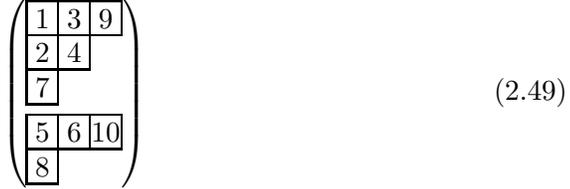
\begin{equation}\label{eg}
\begin{pmatrix}
\diagram{1&3&9\cr2&4\cr7\cr}\\
\diagram{5&6&10\cr8\cr}
\end{pmatrix}
\end{equation}
The {\em addable} and {\em removable nodes} of 
an $l$-multipartition $\bpar$
mean the places a box 
can be added to or removed from its diagram
to again produce a valid 
diagram of an $l$-multipartition.
Because of our convention of viewing the diagram of an $l$-multipartition
as a column vector of Young diagrams, it makes sense to talk about 
one such node being {\em above} or {\em below} another in the diagram.
In the above example, the removable node containing entry $9$ is above the
one containing entry $8$.

Continue with $\La$ fixed as in (\ref{ms}).
Define the {\em $\La$-residue}
of box in the diagram of an $l$-multipartition
to be the integer $m_k + j-i$,
assuming the box
is in the $i$th row and $j$th column of the $k$th Young diagram.
Let ${\mathscr P}^\La$ denote the set of
all $l$-multipartitions $\bpar$ such that the $\La$-residues
of all
the boxes in the diagram of $\bpar$ belong to $I$;
we refer to elements of $\mathscr P^\La$
as {\em $\La$-multipartitions}.
There is then a bijection
\begin{equation}\label{im}
\Col^\La \stackrel{\sim}{\rightarrow} {\mathscr P}^\La,\qquad
A \mapsto \bpar(A)
\end{equation}
defined by letting
$\bpar(A) = (\rpar^{(1)},\dots,\rpar^{(l)})$ denote
$l$-multipartition
such that the parts of $\rpar^{(k)}$ are equal to the entries
in the $k$th column of the entry-wise difference $(A - A^\La)$.

Using the bijection (\ref{im}), it is usually routine
to translate earlier definitions involving column-strict tableaux
into equivalent notions involving multipartitions.
To illustrate this, we explain how to lift the
crystal operators from (\ref{te}) to obtain maps
\begin{equation}
\tilde e_i, \tilde f_i:\mathscr P^\La \rightarrow
\mathscr P^\La \sqcup \{\frownie\}
\end{equation}
Take $\bpar \in \mathscr P^\La$
and $i \in I$. 
Enumerate the 
addable and removable nodes 
of $\La$-residue $i$
in the diagram of 
$\bpar$ 
as $b_1,\dots,b_n$ in order from top to bottom.
Starting from the sequence $(\sigma_1,\dots,\sigma_n)$
in which $\sigma_r$ is $+$ if $b_r$ is addable
or $-$ if $b_r$ is removable, pass to 
the reduced $i$-signature by cancelling $-+$ pairs 
as in $\S$\ref{sscg}.
Then 
$\tilde e_i (\bpar)$ is obtained by removing a box from the
position corresponding to the leftmost $-$ in the reduced $i$-signature,
or $\frownie$ if no such $-$ exists.
Similarly
$\tilde f_i (\bpar)$ is obtained by adding a box to
the position corresponding to the rightmost $+$,
or $\frownie$ if no such $+$ exists.

Under the bijection (\ref{im}), the set 
$\Std^\La$ of standard $\La$-tableaux 
from $\S$\ref{sscomb} corresponds to the
set $\mathscr{R\!P}^\La$ of {\em restricted $\La$-multipartitions},
namely, the multipartitions
$\bpar = (\rpar^{(1)},\dots,\rpar^{(l)}) \in \mathscr P^\La$
such that 
\begin{equation}\label{res}
\rpar^{(i)}_{j+m_i-m_{i+1}} \leq \rpar^{(i+1)}_{j}
\end{equation}
for each $i=1,\dots,l-1$ and $j \geq 1$.
This is the connected component of the crystal graph
generated by the empty multipartition $\varnothing$.
On the other hand the set $\RStd^\La$
from $\S$\ref{stv} corresponds to the set
$\widetilde{\mathscr{R\!P}}^\La$
of multipartitions
$\bpar = (\rpar^{(1)},\dots,\rpar^{(l)}) \in \mathscr P^\La$
such that \begin{equation}\label{reg}
\rpar^{(i)}_j +m_i-m_{i+1}\geq \rpar^{(i+1)}_j
\end{equation}
for each $i=1,\dots,l-1$ and $j \geq 1$.
The maps from (\ref{rect}) and (\ref{rect2}) can be interpreted as mutually 
inverse
bijections between the sets
$\mathscr{R\! P}^\La$ and $\widetilde{\mathscr{R\! P}}^\La$.

\subsection{\boldmath Integral forms and specialization at $q=1$}\label{ssspec}
In the rest of the article we work only at $q=1$.
To explain how to specialize formally,
let $\A := \Z[q,q^{-1}]$ and 
$U_{\!\A}$ denote
Lusztig's $\A$-form for $U$.
This is the $\A$-subalgebra of $U$ generated
by the quantum
divided powers $E_i^{(r)}$, $F_i^{(r)}$,
the elements $D_i, D_i^{-1}$, and the elements
$$
\left[\!\!\!\begin{array}{c}D_i\\r\end{array}\!\!\!\right]
:= \prod_{s=1}^r \frac{D_i q^{1-s} - D_i^{-1} q^{s-1}}{q^s-q^{-s}}
$$
for all $i$ and $r \geq 0$.
The Hopf algebra structure on $U$ makes
$U_{\!\A}$ into a Hopf algebra over $\A$.
The bar involution and the map $\tau$ from (\ref{taudef}) also restrict to 
well-defined maps on $U_{\!\A}$.

There are two natural $\A$-forms for $V(\La)$,
namely, the {\em standard form} $V(\La)_{\A}$
which is obtained by applying $U_{\!\A}$
to the highest weight vector $v_\La$, and the 
{\em costandard form} $V(\La)_{\A}^*$
which is the dual lattice
under the Shapovalov form:
$$
V(\La)_{\A}^* := \left\{v \in V(\La)\:|\:\langle v,w\rangle  \in \A\text{ for all }w \in V(\La)_{\A}\right\}.
$$
The costandard form $V(\La)^*_{\A}$ is the $\A$-submodule
of $V(\La)$
spanned by either the dual-canonical basis or the standard monomial basis.
Moreover, $V(\La)_{\A}$, which is naturally a $U_{\!\A}$-submodule
of $V(\La)^*_{\A}$, is the $\A$-span of the 
quasi-canonical basis $\{P_A\:|\:A \in \Std^\La\}$.
We also have an obvious $\A$-form $F(\La)_{\A}$
for $F(\La)$, namely, the $\A$-span of any of the
four natural bases $\{M_A\}, \{L_A\}, \{T_A\}$ or $\{P_A\}$.
The projection $\pi\left(F(\La)_{\A}\right)$ is precisely the 
costandard form $V(\La)^*_{\A}$ for $V(\La)$.
Similarly there is an $\A$-form 
$\widetilde{F}(\La)_{\A}$
for $\widetilde{F}(\La)$.

With these $\A$-forms in hand, we can finally specialize at $q=1$.
Let $U_\Z$ be the Kostant $\Z$-form for the universal enveloping algebra of
$\mathfrak{gl}_{I_+}(\C)$, with Chevalley generators $e_i, f_i\:(i \in I)$.
Let
\begin{align}
F(\La)_\Z &:= \Z \otimes_{\A} F(\La)_{\A},\label{FZ}\\
\widetilde{F}(\La)_\Z &:= \Z \otimes_{\A} 
\widetilde{F}(\La)_{\A},\label{FZ2}\\
V(\La)_\Z &:= \Z \otimes_{\A} V(\La)_{\A},\label{VZ}\\
V(\La)^*_\Z &:= \Z \otimes_{\A} V(\La)^*_{\A},\label{VZZ}
\end{align}
viewing $\Z$ as an $\A$-module so that $q$ acts as $1$.
These are naturally $U_\Z$-modules;
the divided powers $e_i^{(r)}$
and $f_i^{(r)}$ act as $1 \otimes E_i^{(r)}$
and $1 \otimes F_i^{(r)}$, respectively.
Of course $V(\La)_\Z$ (resp.\ $V(\La)_\Z^*$) 
is the standard (resp.\ costandard) form for the irreducible highest weight
module for $\mathfrak{gl}_{I_+}(\C)$ of highest weight $\La$.

For the remainder of the article we will only be working 
at $q=1$, so it will not cause confusion to use the same notation
$\cM_A, \cL_A, \dots$
for the specializations of the basis
elements $M_A, L_A, \dots$ defined above at $q=1$.
In particular, making this abuse of notation, 
$\{\cS_A\:|\:A \in \Std^\La\}$ becomes precisely the classical standard monomial
basis for $V(\La)^*_\Z$, which can 
be constructed directly as above without going via quantum groups.
The bases
 $\{\cP_A\:|\:A \in \Std^\La\}$
and $\{\cD_A\:|\:A \in \Std^\La\}$ are 
 the usual canonical and dual-canonical bases
of $V(\La)_\Z$ and $V(\La)_\Z^*$, respectively.
Moreover these two bases are dual to each other under
the Shapovalov form (which at $q=1$ coincides with
the contravariant form).

\section{The categorification theorems}

Now we categorify the spaces (\ref{FZ})--(\ref{VZZ})
using parabolic category $\mathcal O$ and degenerate cyclotomic Hecke algebras.
Throughout the section 
we let $I$ be a non-empty set of consecutive integers, and fix all the other notation from
section 2 according to this choice.
Whenever we discuss parabolic category $\mathcal O$ or
appeal to Schur-Weyl duality, we will assume further that
$I$ is bounded-below. 

\subsection{\boldmath Parabolic category $\mathcal O$}\label{spco}
In this subsection, we assume that the index set $I$ is bounded-below.
We want to recall a well known categorification of the $U_\Z$-module
$F(\La)_\Z$ from (\ref{FZ})
using blocks of parabolic category $\mathcal O$
associated to the general linear Lie algebra. 
We refer the reader to \cite[$\S$4.4]{rep} for 
a more detailed treatment.

Since $I$ is bounded-below, there are only finitely many boxes
in the diagram of $\La$. Let $n$ denote this number.
Also let $n_i := m_i+1-\inf(I)$.
This is the number of boxes in the $i$th column of the diagram
of $\La$, so
$$
n = n_1+\cdots+n_l.
$$
We then have simply that
\begin{equation*}
F(\La)_\Z = {\textstyle\bigwedge}^{n_1} V_\Z \otimes\cdots\otimes
{\textstyle\bigwedge}^{n_l} V_\Z
\end{equation*}
where 
$V_\Z$ is the natural $U_\Z$-module (and all tensor products
are over $\Z$).

Consider the Lie algebra 
$\mathfrak{g} := \mathfrak{gl}_n(\C)$,
with its standard Cartan subalgebra $\mathfrak{h}$ 
of diagonal matrices and its standard Borel subalgebra
$\mathfrak{b}$ of upper triangular matrices.
Let $\eps_1,\dots,\eps_n \in \mathfrak{h}^*$ denote the 
standard coordinate functions, so $\eps_i$ picks out the
$i$th diagonal entry of a diagonal matrix.
Given any $\La$-tableau $A$, we 
let
$\cL(A)$ denote the usual irreducible highest weight module for $\mathfrak{g}$
of highest weight
$a_1 \eps_1 + (a_2+1) \eps_2+\cdots + (a_n+n-1) \eps_n$,
where $a_1,\dots,a_n$ is the column-reading
of $A$, i.e. the sequence obtained by reading its entries
down columns starting with the leftmost column.
For $A, B \in \Col^\La$, the 
irreducible modules 
$\cL(A), \cL(B)$ have the same central character
if and only if $\wt(A) = \wt(B)$ according to (\ref{wtdef}).

For $\alpha \in Q_+$, let $\mathcal O^\La_\alpha$ denote the 
category of all $\mathfrak{g}$-modules that are semisimple over $\mathfrak{h}$
and have a composition series with composition factors of the form $\cL(A)$
for $A \in \Col^\La_\alpha$.
When non-zero, $\mathcal O^\La_\alpha$ is known by 
\cite[Theorem 2]{cyclo} to be a single
block of the parabolic
analogue of the Bernstein-Gelfand-Gelfand category $\mathcal O$
associated to the standard parabolic subalgebra $\mathfrak{p}$ of 
$\mathfrak{g}$
with Levi factor
$\mathfrak{gl}_{n_1}(\C) \oplus \cdots \oplus \mathfrak{gl}_{n_l}(\C)$.
Hence
\begin{equation}\label{parc}
\mathcal O^\La := \bigoplus_{\alpha \in Q_+} \mathcal O^\La_\alpha
\end{equation}
is a sum of blocks of this parabolic category $\mathcal O$.
The category $\mathcal O^\La$ is a highest weight category
with irreducible objects $\{\cL(A)\:|\:A \in \Col^\La\}$,
standard objects $\{\cM(A)\:|\:A \in \Col^\La\}$ (which are
parabolic Verma modules), projective indecomposable modules
$\{\cP(A)\:|\:A \in \Col^\La\}$
and indecomposable tilting modules
$\{\cT(A)\:|\:A \in \Col^\La\}$.
The isomorphism classes of these four families of modules
give four natural bases for the Grothendieck group $[\mathcal O^\La]$
of the category $\mathcal O^\La$.

We also have the usual duality
$\circledast$ on $\mathcal O^\La$
defined with respect to the anti-automorphism
$*:\mathfrak{g} \rightarrow \mathfrak{g}$ mapping a matrix to its
transpose.
The irreducible modules are self-dual
in the sense that
\begin{equation}
\cL(A)^\circledast \cong \cL(A)
\end{equation}
for each $A \in \Col^\La$.

As for any highest weight category, the projective indecomposable module $\cP(A)$
has a filtration whose sections are standard modules, and the multiplicity
of the standard module $\cM(B)$ as a section of any 
such filtration, denoted $(\cP(A):\cM(B))$, is
equal to the composition multiplicity $[\cM(B):\cL(A)]$, 
i.e.
\begin{equation}\label{bggrec}
(\cP(A):\cM(B)) = [\cM(B):\cL(A)].
\end{equation}
This result
is usually referred to as {\em Bernstein-Gelfand-Gelfand 
reciprocity} after \cite{BGG}.

The module $\cT(A)$ is the unique (up to isomorphism) self-dual indecomposable
module in $\mathcal O^\La$ possessing a filtration by standard modules
in which $\cM(A)$ appears at the bottom.
There is a twisted version of BGG reciprocity for tilting modules
which describes the multiplicities of standard modules
in any standard filtration of $\cT(A)$.
We refer to this as {\em Arkhipov-Soergel reciprocity};
see $\S$\ref{sark} below for the detailed references.
To formulate the result, given $A \in \Col^\La$, let
$\widetilde{\cL}(A)$ 
denote the irreducible $\mathfrak{g}$-module 
of highest weight
$a_1 \eps_1 + (a_2+1) \eps_2+\cdots + (a_n+n-1) \eps_n$,
where $a_1,\dots,a_n$ is the reverse-column-reading
of the entries of $A$, i.e. the sequence obtained by reading
down columns starting with the rightmost column.
Let 
\begin{equation}
\widetilde{\mathcal O}^\La
:= \bigoplus_{\alpha \in Q_+} 
\widetilde{\mathcal O}^\La_\alpha
\end{equation}
where
$\widetilde{\mathcal O}^\La_\alpha$ is the 
category of all $\mathfrak{g}$-modules that are semisimple over $\mathfrak{h}$
and have a composition series with composition factors of the form $\widetilde{\cL}(A)$
for $A \in \Col^\La_\alpha$.
This is a single
block of the
parabolic category $\mathcal O$
associated to the standard parabolic subalgebra $\widetilde{\mathfrak{p}}$
with Levi factor
$\mathfrak{gl}_{n_l}(\C) \oplus \cdots \oplus \mathfrak{gl}_{n_1}(\C)$
(so $\widetilde{\mathfrak{p}}$ is 
conjugate to the opposite parabolic to $\mathfrak{p}$).
In the highest weight category $\widetilde{\mathcal O}^\La$, 
we have irreducible modules $\{\widetilde{\cL}(A)\:|\:A \in \Col^\La\}$,
standard modules $\{\widetilde{\cM}(A)\:|\:A \in \Col^\La\}$ 
and projective indecomposable modules
$\{\widetilde{\cP}(A)\:|\:A \in \Col^\La\}$.
Now Arkhipov-Soergel reciprocity is the assertion that
\begin{equation}\label{asr}
(\cT(A):\cM(B)) = [\widetilde{\cM}(B): \widetilde{\cL}(A)]
\end{equation}
for all $A, B \in \Col^\La$.

Returning to the discussion just of the category $\mathcal O^\La$,
for each $i \in I$ and $\alpha \in Q_+$
there are some much-studied {\em special projective functors}
(e.g. see \cite{BFK, BKtf, CR})
\begin{equation}\label{spf1}
\mathcal O_{\alpha}^\La
\quad\qquad
\mathcal O_{\alpha+\alpha_i}^\La
\begin{picture}(-10,19)
\put(-46,3){\makebox(0,0){$\stackrel{\stackrel{\;\scriptstyle{\cF_i}_{\phantom{,}}}{\displaystyle\longrightarrow}}{\stackrel{\displaystyle\longleftarrow}{\scriptstyle{\cE_i}}}$}}
\end{picture}
\end{equation}
\vspace{.1mm}

\noindent
defined as follows: $\cF_i$ is defined on
$\mathcal O^\La_{\alpha}$ 
by tensoring with the natural
$\mathfrak{g}$-module of column vectors
then projecting onto $\mathcal O^\La_{\alpha+\alpha_i}$;
 $\cE_i$ is defined on $\mathcal O^\La_{\alpha+\alpha_i}$
by tensoring with the dual of the natural module then projecting onto
$\mathcal O_\alpha^\La$.
Taking the direct sum of these functors over all $\alpha \in Q_+$,
we obtain endofunctors
$\cE_i$ and $\cF_i$ of $\mathcal O^\La$.
These functors $\cE_i$ and $\cF_i$ are biadjoint.
Hence they are both exact, so induce well-defined endomorphisms
of the Grothendieck group $[\mathcal O^\La]$.
We note also that the functors $\cE_i$ and $\cF_i$ commute with the
duality $\circledast$.

Now we can state the following foundational
categorification theorem.
This should be viewed as a translation of
the Kazhdan-Lusztig conjecture for $\mathfrak{g}$
into the language of canonical bases,
and is probably best described as ``folk-lore''
as it seems to have been independently (re-)discovered by 
many different people since the time of \cite{Lubook}.

\begin{Theorem}\label{cat1}
Identify the Grothendieck group
$[\mathcal O^\La]$ with the $U_\Z$-module
$F(\La)_\Z$ by identifying $[\cM(A)]$ with $\cM_A$
for each $A \in \Col^\La$.
\begin{itemize}
\item[(i)]
The following equalities hold for all $A \in \Col^\La$:
\begin{itemize}
\item[(a)]
$[\cL(A)] = \cL_A$;
\item[(b)]
$[\cP(A)] = \cP_A$;
\item[(c)]
$[\cT(A)] = \cT_A$.
\end{itemize}
\item[(ii)]
We have that
$\dim \hom_{\mathfrak{g}}(P, M) = \langle [P], [M] \rangle$
for each $P, M \in \mathcal O^\La$
with $P$ projective.
\item[(iii)]
The endomorphisms of the Grothendieck group
induced by the exact functors $\cE_i$ and $\cF_i$ from (\ref{spf1}) 
coincide with the
action of the Chevalley generators $\cE_i$ and $\cF_i$ of $U_\Z$
for each $i \in I$.
\item[(iv)] For $A \in \Col^\La$ and $i \in I$,
the module $\cE_i (\cL(A))$ (resp.\ $\cF_i (\cL(A))$) 
is non-zero if and only if
$\tilde e_i(A) \neq \frownie$
(resp.\ $\tilde f_i(A) \neq \frownie$), in which case it is
a self-dual indecomposable module
with irreducible socle and head isomorphic
to $\cL(\tilde e_i(A))$
(resp.\ $\cL(\tilde f_i(A))$).
\end{itemize}
\end{Theorem}

\begin{proof}
The assertions (i)(a), (iii) and (iv)
are verified in 
\cite[Theorem 4.5]{rep}.
Note (i)(a) implies that
$$
[\cM(B):\cL(A)] = d_{A,B}(1)
$$
where $d_{A,B}(1)$ is the polynomial from (\ref{t}) evaluated at $q=1$.
Using BGG reciprocity and (\ref{boo}) specialized at $q=1$ we therefore have that
$$
[\cP(A)] = 
\sum_{B \in \Col^\La} d_{A, B}(1) [\cM(B)]
=
\sum_{B \in \Col^\La} d_{A, B}(1) \cM_B
= \cP_A
$$
giving (i)(b).
There is also an analogue of (i)(a)
in the twisted setup proved in \cite[Theorem 4.5]{rep}, so we have that
$$
[\widetilde{\cM}(B): \widetilde{\cL}(A)]
= 
\widetilde{d}_{A,B}(1)
$$
where $\widetilde{d}_{A,B}(1)$ is the polynomial from (\ref{t2}) evaluated
at $q=1$. Combining this with Arkhipov-Soergel reciprocity,
\cite[Corollary 12]{dual} and the definition \cite[(5.8)]{dual}, it follows
that
$$
[\cT(A)] = 
\sum_{B \in \Col^\La} \widetilde{d}_{A,B}(1) [\cM(B)]
= 
\sum_{B \in \Col^\La} \widetilde{d}_{A,B}(1) \cM_B
= \cT_A,
$$
giving (i)(c).
Finally, (ii) is clear from (i)(a), (i)(b),
(\ref{doe}) specialized at $q=1$,
and the observation that
$\dim \hom_{\mathfrak{g}}(\cP(A), \cL(B)) = \delta_{A,B}$.
\end{proof}

In the rest of the section, we are going to
transfer this categorification theorem to 
degenerate cyclotomic Hecke algebras by applying
Schur-Weyl duality for higher levels.
We record two other useful lemmas about parabolic category $\mathcal O$,
the first of which originates in work of Irving. 
It explains the representation theoretic
significance of the subset $\Std^\La$ of $\Col^\La$.

\begin{Lemma}\label{pg0}
For $A \in \Col^\La$ the following are equivalent:
\begin{itemize}
\item[(i)] the projective indecomposable module $\cP(A)$ is 
injective
in $\mathcal O^\La$;
\item[(ii)] $\cP(A)$ is self-dual, i.e. $\cP(A)^\circledast \cong \cP(A)$;
\item[(iii)] $\cL(A)$ embeds into $\cM(B)$
for some $B \in \Col^\La$;
\item[(iv)] $A \in \Std^\La$.
\end{itemize}
\end{Lemma}

\begin{proof}
See \cite[Theorem 4.8]{schur}.
\end{proof}

\begin{Lemma}\label{soc}
For $A \in \Std^\La$ the projective indecomposable module
$\cP(A)$ is isomorphic to the tilting module $\cT(A^\uparrow)$.
Moreover
the standard module $\cM(A^\uparrow)$ has irreducible socle
isomorphic to $\cL(A)$.
\end{Lemma}

\begin{proof}
By Theorem~\ref{cat1}(i) and (\ref{boo}),
the multiplicity $(\cP(A):\cM(B))$
is equal to $d_{A, B}(1)$.
Hence if $A \in \Std^\La$ 
we deduce from Corollary~\ref{yu}
that $\cM(A^\uparrow)$ appears as a section of a standard
filtration of $\cP(A)$ with multiplicity one, 
and all other $\cM(B)$'s arising
satisfy $B < A^\uparrow$.
Since standard filtrations can always be ordered so that the most
dominant section appears at the bottom, this means that
$\cM(A^\uparrow)$ embeds into 
$\cP(A)$. As $\cP(A)$ is self-dual by Lemma~\ref{pg0}, 
we deduce that
$\cP(A) \cong \cT(A^\uparrow)$ by the definition of the latter module.
Moreover $\cP(A)$ has irreducible head $\cL(A)$, hence by
self-duality it also has irreducible socle isomorphic to $\cL(A)$.
Hence $\cM(A^\uparrow)$ must also have irreducible socle isomorphic to $\cL(A)$.
\end{proof}

\subsection{Degenerate cyclotomic Hecke algebras and blocks}\label{ssb}
In this subsection, we allow $I$ to be arbitrary;
from the Hecke algebra point of view the most important case is when $I = \Z$.

For $d \geq 0$, let $H_d$ be the degenerate affine Hecke algebra from \cite{D}.
As a vector space, $$
H_d = \C[x_1,\dots,x_d] \otimes \C S_d,
$$ 
the tensor product
of a polynomial ring in variables $x_1,\dots,x_d$ with the group algebra of the 
symmetric group. Multiplication is defined so that 
$\C[x_1,\dots,x_d] = \C[x_1,\dots,x_d] \otimes 1$ and $\C S_d = 1 \otimes \C S_d$
are subalgebras, and in addition
\begin{equation*}
s_r x_s = x_s s_r \text{ if $s \neq r,r+1$},
\quad
s_r x_{r+1} = x_r s_r + 1,
\quad
x_{r+1} s_r = s_r x_r + 1.
\end{equation*}
Here $s_r$ denotes the $r$th basic transposition.

Suppose we are given $l \geq 0$ and an $l$-tuple
$\bm = (m_1, \cdots, m_l)$ of integers.
Introduce the 
corresponding
{\em degenerate cyclotomic Hecke algebra} of level $l$:
\begin{equation}\label{meeting}
H_d^{\bm}
:= H_d / \langle (x_1-m_1) \cdots (x_1-m_l)\rangle.
\end{equation}
A basic result is that the vectors
\begin{equation}\label{bt}
\{x_1^{r_1} \cdots x_d^{r_d} w\:|\:0 \leq r_1,\dots,r_d < l, w \in S_d\}
\end{equation}
give a basis for $H_d^\bm$, hence 
$\dim H_d^\bm =l^d d!$; see e.g. \cite[Lemma 3.5]{schur}
or \cite[Theorem 7.5.6]{Kbook} for two possible proofs.
Any finite dimensional $H_d^\bm$-module $M$ decomposes into 
{\em weight spaces}
\begin{equation}\label{wtsp}
M = \bigoplus_{\bi \in \Z^d} M_\bi
\end{equation}
where $M_\bi := \{v \in M\:|\:(x_r-i_r)^N v = 0
\text{ for all $r=1,\dots,d$ and $N \gg 0$}\}$.
We have appealed here to a basic fact, namely, 
that the eigenvalues of the $x_r$'s are necessarily integers; see e.g. \cite[Lemma 7.1.2]{Kbook} for a proof.
The
{\em formal character} $\ch M$ of $M$ can then be defined as
\begin{equation}
\ch M := \sum_{\bi \in \Z^d} \dim M_\bi \cdot \bi,
\end{equation}
which is an element of the free $\Z$-module
with basis 
$\{\bi\:|\:\bi \in \Z^d\}$.
Two finite dimensional 
$H_d^\bm$-modules are equal in the Grothendieck group if and only if
their formal characters are equal;
see \cite[Theorem 5.3.1]{Kbook}.

Considering the decomposition (\ref{wtsp}) 
for the regular module, it follows that 
there is a system 
\begin{equation}
\{e(\bi)\:|\:\bi \in \Z^d\}
\end{equation}
of mutually orthogonal idempotents in $H_d^\bm$
such that $e(\bi) M = M_\bi$ for each finite dimensional module $M$.
In fact, each $e(\bi)$
lies in the commutative subalgebra
generated by $x_1,\dots,x_d$,
all but finitely many of the $e(\bi)$'s are necessarily zero,
and their sum is the identity element in $H_d^\bm$.
By \cite[Theorem 1]{cyclo},
the center $Z(H_d^\bm)$ consists of all symmetric polynomials
in $x_1,\dots,x_d$.
Hence, 
the sum of the $e(\bi)$'s for $\bi$ running over any $S_d$-orbit
in $\Z^d$ is either zero or it is a primitive 
central idempotent in $H_d^\bm$.

Now we pick a dominant weight $\La \in P_+$ of level $l$
and take $m_1,\dots,m_l$ to be the integers defined
by the decomposition (\ref{ms}).
Given $\alpha \in Q_+$, let $d := \height(\alpha)$
and let
\begin{equation}
e_\alpha := \sum_{\bi \in I^\alpha} e(\bi) \in H_d^\bm.
\end{equation}
As remarked at the end of the previous paragraph, the non-zero
$e_\alpha$'s are the primitive central idempotents in $H_d^\bm$. 
Define
\begin{equation}\label{fe}
H_\alpha^\La := e_\alpha H_d^\bm.
\end{equation}
The algebra $H^\La_\alpha$ is either zero or it is a single
{\em block} of $H_d^\bm$, and $\{e(\bi)\:|\:\bi \in I^\alpha\}$
is a system of mutually orthogonal idempotents 
summing to the identity in $H^\La_\alpha$.
Finally introduce the algebra
\begin{equation}\label{fe2}
H^\La := \bigoplus_{\alpha \in Q_+} H^\La_\alpha.
\end{equation}
For infinite $I$, this is a locally unital but not unital algebra,
as it is a direct sum of infinitely many
non-zero 
finite dimensional
algebras.
If $I = \Z$ then we have simply that
\begin{equation*}
H^\La = \bigoplus_{d \geq 0} H_d^\bm.
\end{equation*}
However if $I \neq \Z$  then equality does not hold here:
there always exist
blocks of $H_d^\bm$ for sufficiently large $d$
that cannot be realized 
as $H^\La_\alpha$'s for any $\alpha \in Q_+$ (we recall $Q_+$
depends implicitly on the index set $I$).

\subsection{Grothendieck groups, induction and restriction}\label{sgir}
Denote the category 
of finite dimensional 
(resp.\ finitely generated projective)
left 
$H^\La_\alpha$-modules
by $\rep{H^\La_\alpha}$
(resp.\ $\proj{H^\La_\alpha}$).
The Grothendieck group 
$[\rep{H^\La_\alpha}]$ 
(resp. $[\proj{H^\La_\alpha}]$) is the free abelian group with
basis given by the isomorphism classes of irreducible modules
(resp. projective indecomposable modules).
Sometimes it is convenient to work with all blocks at once:
 let $\rep{H^\La}$ (resp.\ $\proj{H^\La}$)
denote 
category of all finite dimensional
(resp.\ finitely generated projective)
locally unital left modules over the locally unital algebra
$H^\La$ from (\ref{fe2}).
The Grothendieck groups of these categories
then decompose into blocks as
\begin{align*}
[\rep{H^\La}] &= \bigoplus_{\alpha \in Q_+} [\rep{H^\La_\alpha}],\\
[\proj{H^\La}] &= \bigoplus_{\alpha \in Q_+} [\proj{H^\La_\alpha}].
\end{align*}
We have the {\em Cartan pairing}
\begin{equation}\label{cp}
\langle.,.\rangle:[\proj{H^\La}] \times
[\rep{H^\La}] \rightarrow \Z
\end{equation}
defined by setting
$\langle [P], [M] \rangle :=
\dim \hom_{H^{\La}} (P, M)$
for $P \in \proj{H^\La}$ and $M \in \rep{H^\La}$.
Of course different blocks are orthogonal under this pairing.

There is an obvious embedding
$H_d \hookrightarrow H_{d+1}$
which in view of (\ref{bt})
factors through the quotients to induce an embedding
$\iota:H_d^\bm \hookrightarrow H_{d+1}^\bm$.
This map has the property that
\begin{equation}\label{uprop}
\iota(e(i_1,\dots,i_d)) = \sum_{i \in \Z} e(i_1,\dots,i_d,i)
\end{equation}
for any $i_1,\dots,i_d \in \Z$.
Given $i \in I$ and $\alpha \in Q_+$ of height $d$,
we can compose $\iota$
on one side with the embedding
$H^\La_\alpha  \hookrightarrow H_d^\bm$
and on the other side with the projection
$H_{d+1}^\bm \twoheadrightarrow H^\La_{\alpha+\alpha_i}$
defined by multiplication by the 
block idempotent $e_{\alpha+\alpha_i}$, to obtain a non-unital
algebra homomorphism
\begin{equation}
\iota_{\alpha,\alpha_i}:
H^\La_\alpha \rightarrow
H^\La_{\alpha+\alpha_i}.
\end{equation}
By (\ref{uprop}), the map $\iota_{\alpha,\alpha_i}$
maps the identity element $e_{\alpha}$ of $H^\La_\alpha$
to the idempotent
\begin{equation}
e_{\alpha,\alpha_i} := \sum_{\bi \in I^{\alpha+\alpha_i}\
i_{d+1}=i} e(\bi) \in H^\La_{\alpha+\alpha_i}.
\end{equation}

To the homomorphism $\iota_{\alpha,\alpha_i}$, we associate the induction and restriction
functors
\begin{equation}\label{EF1}
\rep{H^\La_\alpha}
\quad\qquad
\rep{H^\La_{\alpha+\alpha_i}}.
\begin{picture}(-10,19)
\put(-78,3){\makebox(0,0){$\stackrel{\stackrel{\;\scriptstyle{\cF_i}_{\phantom{,}}}{\displaystyle\longrightarrow}}{\stackrel{\displaystyle\longleftarrow}{\scriptstyle{\cE_i}}}$}}
\end{picture}
\end{equation}
\vspace{.1mm}

\noindent
So on an object $M \in \rep{H^\La_{\alpha+\alpha_i}}$
we define
$\cE_i M$ to be the 
the vector space $e_{\alpha,\alpha_i} M$ 
viewed as an $H^\La_\alpha$-module
via the homomorphism $\iota_{\alpha,\alpha_i}$,
and on an object
$M \in \rep{H^\La_\alpha}$ we define
$\cF_i M$ to be $H^\La_{\alpha+\alpha_i} e_{\alpha,\alpha_i}
\otimes_{H^\La_\alpha} M$.
Taking the direct sum of these functors over all $\alpha \in Q_+$,
we obtain endofunctors
$\cE_i$ and $\cF_i$ of
$\rep{H^\La}$.
It is known that these functors $\cE_i$ and $\cF_i$ are biadjoint; 
see \cite[Lemma 8.2.2]{Kbook}.
Hence they induce endomorphisms of the Grothendieck groups
$[\rep{H^\La}]$ and $[\proj{H^\La}]$, and
these induced endomorphisms are biadjoint with respect to 
the {Cartan pairing}.

We note finally that $H_d^\bm$ possesses an anti-automorphism
$*$ with $s_r^* = s_r$ and $x_r^* = x_r$ for each $r$. 
This anti-automorphism leaves
each $H^\La_\alpha$ invariant so gives an anti-automorphism
\begin{equation}\label{Star}
*:H^\La_\alpha \rightarrow H^\La_\alpha
\end{equation}
for each $\alpha \in Q_+$.
Using this we can introduce a natural duality $\circledast$ on finite dimensional left 
$H_\alpha^\La$-modules: $M^\circledast$ denotes the dual space with left
action defined via the anti-automorphism $*$. This duality fixes the irreducible modules. The functors $\cE_i$ and $\cF_i$ commute with $\circledast$; see \cite[Lemma 8.2.2]{Kbook} again.

\subsection{Specht modules}
We want to define one important family of $H^\La_\alpha$-modules,
namely, the {\em Specht modules} 
\begin{equation}
\left\{\cS(A)\:\big|\:
A \in \Col^\La_\alpha\right\}
\end{equation}
following \cite[(6.20)]{schur}; this approach to defining Specht modules
by induction from level one goes back at least to Vazirani.
Given $d=d'+d''$, we have the obvious natural embedding of
$H_{d',d''} := H_{d'} \otimes H_{d''}$ into $H_d$.
If $M'$ is an $H_{d'}$-module and $M''$ is an $H_{d''}$-module, we 
define
\begin{equation}\label{ip}
M' \circ M'' := H_d \otimes_{H_{d',d''}} (M' \boxtimes M''),
\end{equation}
where $M' \boxtimes M''$ denotes the outer tensor product.

For any partition $\rpar$ of $d$ let $\cS(\rpar)$ denote the corresponding
irreducible $\C S_d$-module.
For any $m \in \C$ there is an evaluation homomorphism
\begin{equation}\label{evh}
\ev_m:H_d \twoheadrightarrow \C S_d
\end{equation}
mapping each $s_r$ to $s_r$ and
$x_1$ to $m$. Let $\ev_m^*(\cS(\rpar))$ denote the lift of $\cS(\rpar)$
to an $H_d$-module via this homomorphism $\ev_m$.
Now suppose that $A \in \Col^\La_\alpha$ and let
$\bpar(A) = (\rpar^{(1)},\dots,\rpar^{(l)})$ be the corresponding
$\La$-multipartition according to the bijection (\ref{im}).
Then we define
\begin{equation}\label{spcht}
\cS(A) := 
\ev_{m_1}^* (\cS(\rpar^{(1)})) \circ \cdots \circ \ev_{m_l}^* (\cS(\rpar^{(l)})).
\end{equation}
As written, this is a finite dimensional left $H_d$-module, but 
it is not hard to check that $(x_1-m_1) \cdots (x_1-m_l)$ acts as zero, 
hence $\cS(A)$ 
can naturally be viewed as an $H_d^\bm$-module.
A straightforward character calculation (see e.g. \cite[p.247]{cyclo})
shows moreover that 
$\cS(A)$ belongs to the block $e_{\alpha}$, i.e.
$e_{\alpha} \cS(A) = \cS(A)$. Hence
$\cS(A)$ is a unital $H^\La_\alpha$-module.

\begin{Remark}\rm
The Specht module $\cS(A)$ just defined
is the same as the cell module $\cS(\bpar)$
that is defined in \cite[$\S$6]{AMR}
for the multipartition $\bpar := \bpar(A)$.
This is the degenerate analogue
of the Specht module of
\cite[(3.28)]{DJM}.
One could also choose to work everywhere in terms of the modules
\begin{equation}\label{dualspecht}
\widetilde{\cS}(A) :=
\ev_{m_l}^*(\cS(\rpar^{(l)}))
\circ \cdots \circ \ev_{m_1}^*(\cS(\rpar^{(1)})),
\end{equation}
using the same notation as (\ref{spcht}).
Note by \cite[Corollary 3.7.5]{Kbook} that
\begin{equation}\label{Aredual}
\widetilde{\cS}(A) \cong \cS(A)^\circledast.
\end{equation}
In particular, this means that
$\cS(A)$ and $\widetilde{\cS}(A)$ have the same formal characters,
so are equal in the Grothendieck group.
\end{Remark}

\subsection{\boldmath Schur-Weyl duality for level $l$}\label{sswd}
In this subsection we assume once more that $I$ is bounded-below.
We are going to recall 
the Schur-Weyl duality from \cite{schur}
which links
the representation theory of the algebra $H^\La$ from (\ref{fe2})
to  the parabolic category $\mathcal O^\La$ 
from (\ref{parc}).
Let $\cT_d$ denote the $d$th tensor power of the natural
$\mathfrak{g}$-module of column vectors.
For any $\mathfrak{g}$-module $M$, 
there is a natural right action of 
the degenerate affine Hecke algebra $H_d$
on $M \otimes \cT_d$ commuting with the left action of 
$\mathfrak{g}$.
This action is defined by letting
elements of $S_d$ act by permuting tensors in $\cT_d$
like in classical Schur-Weyl duality.
The remaining 
generator $x_1$ acts as 
$\Omega \otimes 1^{\otimes (d-1)}$,
where $\Omega := \sum_{i,j=1}^{n} e_{i,j}\otimes e_{j,i} \in \mathfrak{g}
\otimes \mathfrak{g}$ (the ``trace form'').

Now recall the ground-state tableau $A^\La$ from 
$\S$\ref{sscomb}. 
It is the only column-strict $\La$-tableau of weight $\La$,
hence $\cP(A^\La)\cong \cM(A^\La) \cong \cL(A^\La)$.
This module plays a very special role:
according to Theorem~\ref{cat1}
its isomorphism class 
in the Grothendieck group
$[\mathcal O^\La]$ coincides with the highest weight vector
$v_\La\in F(\La)_\Z$.
For $\alpha \in Q_+$ with $\height(\alpha) = d$, we define
$\cT^\La_\alpha$ to be the projection of
the tensor product $\cP(A^\La) \otimes \cT_d$ onto the
block $\mathcal O^\La_\alpha$.
Equivalently, recalling the 
functor $\cF_i$ from (\ref{spf1}), we have that
\begin{equation}
\cT_\alpha^\La = \bigoplus_{\bi \in I^\alpha} 
\cF_{i_d} \cdots \cF_{i_1}( \cP(A^\La) ).
\end{equation}
Then set
\begin{equation}\label{all}
\cT^\La := \bigoplus_{\alpha \in Q_+} \cT^\La_\alpha.
\end{equation}
By a {\em prinjective module} we mean a module that is both 
projective and injective.
By Lemma~\ref{pg0}, we know that 
the modules $\{\cP(A)\:|\:A \in \Std^\La\}$
give a full set of representatives for the isomorphism classes
of
prinjective indecomposable modules
in the category $\mathcal O^\La$.

\begin{Lemma}\label{pg}
The module $\cT^\La$ is a prinjective generator for $\mathcal O^\La$,
i.e. it is a direct sum of
copies of the prinjective indecomposable modules
$\{\cP(A)\:|\:A \in \Std^\La\}$, and each of these
modules appears as a summand of $\cT^\La$ with multiplicity at least one.
\end{Lemma}

\begin{proof}
We note that $\cP(A^\La)$
is prinjective, and tensoring with finite dimensional
modules preserves projective and injective modules.
Hence each $\cP(A^\La) \otimes \cT_d$ is prinjective, so all their
summands are too.
This shows that all indecomposable summands of $\cT^\La$
are of the form $\cP(A)$ for some $A \in \Std^\La$.
The fact that each of these appears in $\cT^\La$
with multiplicity
at least one is established in \cite[Corollary 4.6]{schur}.
\end{proof}

Now we discuss the 
endomorphism algebra of $\cT^\La_\alpha$ for $\alpha \in Q_+$.
Letting $d := \height(\alpha)$,
\cite[Corollary 5.12]{schur} shows that
the natural right action of $H_d$ on $\cP(A^\La) \otimes \cT_d$ defined above
factors through the quotient to make $\cP(A^\La) \otimes \cT_d$
into a well-defined right $H_d^\bm$-module.
Moreover by \cite[Lemma 5.8]{cyclo} 
the canonical $\mathfrak{g}$-equivariant projection
of $\cP(A^\La) \otimes \cT_d$ onto the summand $\cT_\alpha^\La$
coincides with the action of
the central idempotent
$e_{\alpha} \in H_d^\bm$.
Hence $\cT_\alpha^\La$ is naturally a $(U(\mathfrak{g}), H_\alpha^\La)$-bimodule.
The following theorem should be viewed as a slight 
reformulation of the main result of
Schur-Weyl duality for level $l$ from \cite{schur};
it was formulated on a block-by-block basis like this already
in \cite[Theorem 5.9]{cyclo}.

\begin{Theorem}\label{swd0}
The natural right action of $H_\alpha^\La$ on $\cT^\La_\alpha$ induces an algebra isomorphism
between 
$H^\La_\alpha$ and
$\End_{\mathfrak{g}}(\cT^\La_\alpha)^{\op}$.
\end{Theorem}

\begin{proof}
This is a consequence of \cite[Theorem 5.13]{schur},
\cite[Corollary 6.7]{schur}
and the 
description of the idempotent $e_{\alpha}$
from \cite[Lemma 5.8]{cyclo} mentioned just before the statement of the theorem.
\end{proof}

Recalling (\ref{all}) and (\ref{fe2}), we have that
$\cT^\La$ is a $(U(\mathfrak{g}), H^\La)$-bimodule,
and Theorem~\ref{swd0} allows us to identify
\begin{equation}\label{dc}
H^\La = \End_{\mathfrak{g}}(\cT^\La)^{\op}.
\end{equation}
Now introduce the {\em Schur functors}
\begin{align}\label{sf1}
\pi :=
 \hom_{\mathfrak{g}}(\cT^\La, ?)&:
\mathcal O^\La
\rightarrow \rep{H^\La},\\
\pi^*:=
\cT^\La \otimes_{H^\La} ?&:
\rep{H^\La}\rightarrow \mathcal O^\La.\label{sf2}
\end{align}
Because of (\ref{dc}), 
the functor $\pi$ is a
{\em quotient functor} in the general sense of \cite[$\S$III.1]{Gab}, 
so we are in a well understood situation.
In particular, the functor $\pi$ is exact
and $\pi^*$ is left adjoint to $\pi$, hence $\pi^*$
sends projectives to projectives.
The following theorem collates all the other important known facts
about these functors that we will need later on.

\begin{Theorem}\label{swd}
For $A \in \Col^\La$,
$\pi(\cL(A))$ is non-zero if and only if $A \in \Std^\La$.
Moreover:
\begin{itemize}
\item[(i)]
The modules $\{\pi(\cL(A))\:|\:A \in \Std^\La\}$
give a complete set of pairwise inequivalent irreducible $H^\La$-modules.
\item[(ii)]
For $A \in \Std^\La$, the module $\pi(\cP(A))$ is the projective cover of
$\pi (\cL(A))$, and $\pi^*(\pi(\cP(A))) \cong \cP(A)$.
\item[(iii)] For $A \in \Col^\La$,
we have that
$\pi (\cM(A)) \cong \cS(A)$.
\item[(iv)] 
There is an
isomorphism of functors 
$\cF_i \circ \pi \cong \pi \circ \cF_i$,
where $\cF_i$ on the left hand side comes from (\ref{EF1})
and $\cF_i$ on the right hand side comes from (\ref{spf1}).
\item[(v)]
For any $M, N \in \mathcal O^\La$ such that 
all composition factors of the head of $M$ and the socle of $N$
are of the form $\cL(A)$ for $A \in \Std^\La$, the functor
$\pi$ defines a vector space isomorphism
$$
\hom_{\mathfrak{g}}(M, N) \stackrel{\sim}{\rightarrow}
\hom_{H^\La}(\pi(M), \pi(N)).
$$
\item[(vi)] 
The functor $\pi$ is fully faithful both on projectives
and on tilting modules.
\end{itemize}
\end{Theorem}

\begin{proof}
The first statement is clear from Lemma~\ref{pg}.
Given that, (i) and (v) are general facts about quotient functors;
see \cite[Corollary 3.1e]{GL}
and \cite[Corollary 3.1c]{GL}, respectively.
Also (iii) is \cite[Theorem 6.12]{schur},
(iv) follows from \cite[Lemma 5.16]{schur}.

For (ii), 
let $e \in \End_{\mathfrak{g}}(\cT^\La)^{\op}$ be an idempotent
such that $(\cT^\La) e \cong \cP(A)$;
such idempotents exist by Lemma~\ref{pg}.
By (\ref{dc}) and Fitting's lemma, $e$ is a primitive idempotent
in $H^\La$, hence $(H^\La) e$ is a projective indecomposable module.
Clearly this is isomorphic to $\pi (\cP(A))$,
which maps surjectively onto $\pi(\cL(A))$.
Hence $\pi(\cP(A))$ is the projective cover of $\pi(\cL(A))$.
It remains to show that $$
\pi^*(\pi(\cP(A))) \cong \cP(A).
$$
Since $\cP(A)$ has a standard filtration,
Lemma~\ref{pg0} implies that all constituents of the socle of
$\cP(A)$ are of the form $\cL(B)$ for $B \in \Std^\La$.
Using this, \cite[Theorem 3.1d]{GL} and 
\cite[Lemma 3.1f]{GL} we get that
$\pi^* (\pi(\cP(A))) \cong \pi^* ((H^\La)e)
\cong (\cT^\La)e \cong \cP(A)$.

Finally, for (vi), the fact that $\pi$ is fully faithful
on projectives is \cite[Theorem 6.10]{schur}.
The fact that it is fully faithful on tilting modules
is a consequence of (v), since all composition factors of 
the socle and head of any tilting module
are of the form $\cL(A)$ for $A \in \Std^\La$.
For the latter statement, tilting modules are self-dual, so
it suffices to verify it for the socle.
Now use the fact that tilting modules 
have a filtration by standard modules, and all composition factors of 
the socle
of any standard module are of the form $\cL(A)$ for $A \in \Std^\La$
according to Lemma~\ref{pg0}(iii).
\end{proof}

\subsection{The main categorification theorem}\label{stmr}
Now we use Schur-Weyl duality
to deduce the degenerate analogue of Ariki's categorification theorem
from Theorem~\ref{cat1},
working now with an arbitrary index set $I$.
The first job is to recover the classification of the
irreducible $H^\La$-modules. The same strategy was
noted already in \cite[Theorem 6.15]{schur}.

\begin{Theorem}\label{class}
For $A \in \Std^\La$, the Specht module
$\cS(A)$ has irreducible head denoted $\cD(A)$.
The modules
$\{\cD(A)\:|\:A \in \Std^\La\}$
give a full set of pairwise inequivalent irreducible $H^\La$-modules.
\end{Theorem}

\begin{proof}
Suppose first that $I$ is bounded-below and let $\pi$
be the Schur functor from (\ref{sf1}).
By Theorem~\ref{swd}(i) the modules
$\{\pi(\cL(A))\:|\:A \in \Std^\La\}$ give a complete set of irreducible
modules. Moreover by Theorem~\ref{swd}(iii)
we know that $\pi(\cM(A)) \cong \cS(A)$.
Therefore we just need to show that
$\pi(\cM(A))$ has irreducible head isomorphic to $\pi(\cL(A))$.
This follows by Theorem~\ref{swd}(v) and the fact that $\cM(A)$
has irreducible head isomorphic to $\cL(A)$:
for any $B \in \Std^\La$
$$
\hom_{H^\La}(\pi(\cM(A)), \pi(\cL(B)))
\cong
\hom_{\mathfrak{g}}(\cM(A), \cL(B))
$$
which is one dimensional if $B = A$ and zero otherwise.

Now assume that $I$ is not bounded-below.
It is enough to show for each $\alpha \in Q_+$ 
that 
$\cS(A)$ has irreducible head denoted $D(A)$
for each $A \in \Std^\La_\alpha$, and that
the irreducible modules
$\{\cD(A)\:|\:A \in \Std^\La_\alpha\}$ give a complete set
of irreducible $H^\La_\alpha$-modules.
Given $\alpha$, we pick $k$ such that $\alpha$
is a sum of simple roots of the form $\alpha_i$ for $i \geq k$.
Then the block $H^\La_\alpha$ already appears as a block of
the algebra $H^\La$ 
defined with respect to the bounded-below
index set $I_{\geq k}$. Moreover for $A \in \Std^\La_\alpha$,
it is clear from the definition that
the Specht module
$\cS(A)$ is the same module regardless of whether we are considering
$I$ or $I_{\geq k}$.
Hence we are done by the previous paragraph.
\end{proof}

\begin{Remark}\rm\label{ap}
There is another natural way to parametrize irreducible modules in terms of the set
$\RStd^\La$ of reverse-standard
$\La$-tableaux from $\S$\ref{stv}.
To see this alternative parametrization, note
for $A \in \RStd^\La$ (and bounded-below $I$)
that
the standard module $\cM(A)$ 
has irreducible socle
isomorphic to $\cL(A^{\downarrow})$;
see Lemma~\ref{soc} and (\ref{rect2}).
Using this and arguing exactly like in the proof of Theorem~\ref{class},
it follows (for arbitrary $I$) that 
the Specht module $\cS(A)$ has irreducible socle 
denoted $\widetilde{D}(A)$ for each $A \in \RStd^\La$.
Moreover
\begin{equation}\label{trans}
\widetilde{D}(A) \cong D(A^{\downarrow}).
\end{equation}
Applying the natural duality and recalling (\ref{dualspecht})--(\ref{Aredual}), 
we deduce that the  
module $\widetilde{S}(A) \cong S(A)^\circledast$
has irreducible head $\widetilde{D}(A)$
for each $A \in \RStd^\La$, and the irreducible modules
$\{\widetilde{D}(A)\:|\:A \in \RStd^\La\}$
give a complete set of pairwise inequivalent irreducible $H^\La$-modules.
\end{Remark}

For each $A \in \Std^\La$, 
let $\cY(A)$ denote the projective cover of $\cD(A)$;
 if $I$ is bounded-below then we can take $\cY(A) := \pi(\cP(A))$
according to Theorem~\ref{swd}(ii).
Then
$\{[\cD(A)]\:|\:A \in \Std^\La\}$
and
$\{[\cY(A)]\:|\:A \in \Std^\La\}$
give canonical bases
for the Grothendieck groups
$[\rep{H^\La}]$ and $[\proj{H^\La}]$, respectively.
These two bases  are dual to each other under the Cartan pairing
from (\ref{cp}).
In the statement of the following theorem, the index set $I$ is arbitrary.

\begin{Theorem}\label{cat2}
Identify $[\rep{H^\La}]$ (resp.\ $[\proj{H^\La}]$)
with the costandard form $V(\La)_\Z^*$ 
(resp.\ the standard form $V(\La)_\Z$) 
for the irreducible $\mathfrak{gl}_{I_+}(\C)$-module
 of highest weight $\La \in P_+$
by identifying $[\cD(A)]$ with 
the dual-canonical basis vector $\cD_A$ 
(resp.\ $[\cY(A)]$ with the canonical basis vector $\cP_A$)
for each $A \in \Std^\La$.
\begin{itemize}
\item[(i)]
The map $[\proj{H^\La}] \rightarrow [\rep{H^\La}]$
induced by the natural embedding of $\proj{H^\La}$
as a full subcategory $\rep{H^\La}$ is identified with the inclusion
$V(\La)_\Z \hookrightarrow V(\La)_\Z^*$. In particular, this
map is injective.
\item[(ii)] We have that $[\cS(A)]=\cS_A$ 
for each $A \in \Col^\La$.
In particular, the isomorphism classes
$\{[\cS(A)]\:|\:A \in \Std^\La\}$ give a basis for $[\rep{H^\La}]$
which coincides with the standard monomial basis
for $V(\La)_\Z^*$.
\item[(iii)]
For each $B \in \Col^\La$ we have that
$$
[\cS(B)] = \sum_{A \in \Std^\La} d_{A,B}(1) [\cD(A)]
$$
where $d_{A,B}(1)$ is the
coefficient of $\cM_B$
when the canonical basis element $\cP_A$ is expanded 
in terms of the monomial basis 
of $F(\La)_\Z$; see (\ref{t}).
\item[(iv)] The Cartan pairing from (\ref{cp})
is identified with the Shapovalov pairing
$\langle.,.\rangle:V(\La)_\Z \times V(\La)_\Z^* \rightarrow \Z$.
\item[(v)] The endomorphisms 
of $[\rep{H^\La}]$ and $[\proj{H^\La}]$
induced by the
exact functors $\cE_i$ and $\cF_i$ from (\ref{EF1})
coincide with the actions of the Chevalley generators of $U_\Z$.
\item[(vi)] For $A \in \Std^\La$ and $i \in I$, the module
$\cE_i (\cD(A))$ (resp.\ $\cF_i (\cD(A))$) is non-zero if and only if
$\tilde e_i(A) \neq \frownie$ (resp.\ $\tilde f_i(A) \neq \frownie$),
in whch case it is a self-dual indecomposable module
with irreducible socle and head
isomorphic to $\cD(\tilde e_i(A))$ (resp.\ $\cD(\tilde f_i(A))$).
\end{itemize}
\end{Theorem}

\begin{proof}
It suffices to prove the theorem in the special case that $I$
is bounded-below.
Assuming that is the case from now on, 
we first claim that the following diagram of $\Z$-linear
maps commutes:
$$
\begin{CD}
[\proj{H^\La}]&@>\pi^*>>&
[\mathcal O^\La]&@>\pi>>&[\rep{H^\La}]\\
@VVV&&@VVV&&@VVV\\
V(\La)_\Z&@>>\pi^*>&
F(\La)_\Z&@>>\pi>&V(\La)_\Z^*.
\end{CD}
$$
Here the top maps $\pi$ and $\pi^*$ 
are induced by the Schur functors
from (\ref{sf1})--(\ref{sf2}).
The bottom $\pi^*$ is the natural inclusion and the bottom
$\pi$ is induced
by the projection from (\ref{pie}). 
The middle  
vertical map is the identification from Theorem~\ref{cat1}
which sends $[\cL(A)]$ to $\cL_A$
and $[\cP(A)]$ to $\cP_A$ for each $A \in \Col^\La$. 
The left (resp.\ right) hand vertical 
hand map is the
identification from the statement of this theorem which maps
$[\cY(A)]$ to $\cP_A$ (resp.\ $[\cD(A)]$ to $\cD_A$) for $A \in \Std^\La$. 
To check that the right hand square commutes, 
it suffices to check it commutes on
$[\cL(A)]$ for each $A \in \Col^\La$. This follows 
because
the top map sends $[\cL(A)]$ to $[\cD(A)]$ if
$A \in \Std^\La$ and to zero otherwise by Theorem~\ref{swd},
while
the bottom map sends $\cL_A$ to $\cD_A$ if $A \in \Std^\La$
and to zero otherwise by Theorem~\ref{las}.
To check the left hand square commutes we
check it on $[\cY(A)]$ for each $A \in \Std^\La$;
this follows because by Theorem~\ref{swd}(ii) we have that 
$[\pi^* (\cY(A))] = [\cP(A)]$.
This establishes the claim.

The composite map along the top of the above diagram
maps $[\cY(A)]$ to $[\cY(A)]$ by Theorem~\ref{swd}(ii).
Hence it is exactly the map induced by the embedding of
$\proj{H^\La}$ into $\rep{H^\La}$.
The composite map along the bottom of the above diagram is the
natural inclusion of $V(\La)_\Z$ into $V(\La)^*_\Z$.
This establishes (i).
Next we compute the image of $[\cM(A)]$ for $A \in \Col^\La$ 
going both ways around the right hand square in the above diagram.
The top map $\pi$ sends $[\cM(A)]$ 
to $[\cS(A)]$ by Theorem~\ref{swd}(iii),
while the left and bottom maps send $[\cM(A)]$ to $\cM_A$ then to $\cS_A$.
Hence $[\cS(A)]$ is identified with $\cS_A$.
Recalling the standard monomial basis theorem from (\ref{smb})
this proves (ii). Also
(iii) now follows from (\ref{sind}) and (\ref{boo}).
For (iv) we just note that
$\{[\cY(A)]\:|\:A \in \Std^\La\}$ and
$\{[\cD(A)\:|\:A \in \Std^\La\}$ are dual bases under the Cartan pairing.
So we are done immediately as
$\{\cP_A\:|\:A \in \Std^\La\}$ and
$\{\cD_A\:|\:A \in \Std^\La\}$ are dual bases under the
Shapovalov form according to Theorem~\ref{maincb}.
Now (v) for $\cF_i$ follows from Theorem~\ref{cat1}(iii)
and Theorem~\ref{swd}(iv).
We then get it for $\cE_i$ too using (iv) and adjointness.

Finally for (vi), it is enough to check it in the case of $\cF_i$,
since it then follows easily for $\cE_i$ too by an
adjointness argument.
Since $\cF_i$ commutes with duality, we know at once that
$\cF_i \cD(A)$ is self-dual.
It remains to compute its head. So we need to compute
$\hom_{H^\La} (\cF_i \cD(A), \cD(B))$
for $A, B \in \Std^\La$.
By Theorem~\ref{swd}(v),
$\cF_i \cD(A) \cong \pi (\cF_i \cL(A))$.
The modules $M := \cF_i \cL(A)$ and $N := \cL(B)$
satisfy the conditions from Theorem~\ref{swd}(vi)
thanks to Theorem~\ref{cat1}(vi).
Hence we get that
$$
\hom_{H^\La} (\cF_i \cD(A), \cD(B))
\cong \hom_{\mathfrak{g}}(\cF_i \cL(A), \cL(B)).
$$
By Theorem~\ref{cat1}(vi) this is zero unless
$\tilde f_i(A) = B$, when it is one dimensional.
This implies (vi).
\end{proof}

\begin{Remark}\rm
We stress that only the proof not the statement of
Theorem~\ref{cat2} should be regarded as new.
For example, as we said in the introduction, 
many parts of the theorem 
can be deduced from the generic case of Ariki's original theorem 
from \cite{Ari} using \cite[Corollary 2]{young} and \cite[$\S$6]{AMR}.
Some aspects of the theorem also appear in Grojnowski's
paper \cite{Groj}; see \cite[Theorems 9.5.1, 10.3.5]{Kbook}.
More general results than (ii) and (vi) have been proved in
the root of unity case by Jacon in \cite{Jac}
and Ariki in \cite{Abranch}, respectively.
\end{Remark}

\subsection{``Tensoring with sign''}
In this subsection we assume that $I = \Z$.
There is an obvious $\Z$-linear 
map $t:P \rightarrow P$
mapping $\La_i$ to $\La_{-i}$
and $\alpha_i$ to $\alpha_{-i}$ for each $i \in \Z$.
We let
\begin{equation}
\La^t := \La_{-m_l}+\cdots+\La_{-m_1}
\end{equation}
denote the image of $\La$ under this map; note that
$-m_l \geq \cdots \geq -m_1$.

We are interested in the {\em sign automorphism}
\begin{equation}
\sigma:H_d \rightarrow H_d, \qquad
s_r \mapsto -s_r,\:x_r \mapsto -x_r.
\end{equation}
For an $H_d$-module $M$, we write $\sigma^*(M)$ for the $H_d$-module
obtained from $M$ by twisting the action by this automorphism $\sigma$.
Recalling (\ref{evh}), we note for any $\C S_d$-module $M$ that
\begin{equation}\label{sign}
\sigma^*(\ev_m^*(M)) \cong \ev_{-m}^*(M \otimes \operatorname{sgn})
\end{equation}
where $M \otimes \operatorname{sgn}$ denotes the $\C S_d$-module obtained by 
tensoring $M$ with the one dimensional sign representation.
Moreover,
\begin{align}
\sigma^*(M^{\circledast}) &\cong \sigma^*(M)^\circledast,\label{sign1}\\
\sigma^*(M \circ N) &\cong \sigma^*(M) \circ \sigma^*(N),\label{sign2}
\end{align}
where $\circ$ is the induction product from (\ref{ip}).

Suppose also that we are given $\alpha \in Q_+$ of height $d$.
Recalling (\ref{meeting}), let 
\begin{equation}
H_d^{\bm^t}
:= H_d / \langle (x_1+m_l) \cdots (x_1+m_1)\rangle.
\end{equation}
Notice that the automorphism 
$\sigma$ factors through the quotients to induce 
an isomorphism
$H_d^\bm \stackrel{\sim}{\rightarrow} H_d^{\bm^t}$.
This isomorphism maps the idempotent
$e(i_1,\dots,i_d)$ to the idempotent
$e(-i_1,\dots,-i_d)$, hence it sends the block idempotent
$e_{\alpha}$ to $e_{\alpha^t}$.
This establishes that $\sigma$ induces an isomorphism
\begin{equation}\label{sigmalocal}
\sigma:H^\La_\alpha \stackrel{\sim}{\rightarrow} H^{\La^t}_{\alpha^t}.
\end{equation}
If $M$ is an $H^{\La^t}_{\al^t}$-module
then $\sigma^*(M)$ is naturally an $H^\La_\alpha$-module.

We define the {\em transpose} $\la^t$ of an $l$-multipartition
$\la = (\la^{(1)},\dots,\la^{(l)})$ to be the $l$-multipartition
$((\la^{(l)})^t,\dots,(\la^{(1)})^t)$,
where $(\la^{(r)})^t$ is the usual transpose of the partition
$\la^{(r)}$.
There is a bijection
\begin{equation}\label{bij}
\Col^\La_\alpha \stackrel{\sim}{\rightarrow} \Col^{\La^t}_{\alpha^t},
\qquad
A \mapsto A^t
\end{equation}
where for  a $\La$-tableau $A$, we write $A^t$ for 
the unique $\La^t$-tableau
that satsifies
$\la(A^t) = \la(A)^t$, recalling the notation (\ref{im}).
We stress that this depends implicitly on the fixed choice
of $\La$.

\begin{Lemma}\label{bij2}
The bijection (\ref{bij}) restricts to a bijection
$\RStd^\La_\alpha\stackrel{\sim}{\rightarrow} \Std^{\La^t}_{\alpha^t}$.
\end{Lemma}

\begin{proof}
Using (\ref{res})--(\ref{reg}) and the definitions, this reduces to checking
for partitions $\la$ and $\mu$ 
and an integer $m \geq 0$ that
$\la_j+m \geq \mu_j$ for all $j \geq 1$ if and only if
$\mu_{j+m}^t \leq \lambda_j^t$ for all $j \geq 1$.
This is an easy combinatorial exercise.
\end{proof}

\begin{Lemma}\label{ssig}
For $A \in \Col^\La_\alpha$,
we have that $\sigma^*(\cS(A^t)) \cong \widetilde{\cS}(A)
\cong \cS(A)^\circledast$.
\end{Lemma}

\begin{proof}
Recalling the definition (\ref{spcht}),
we have by (\ref{sign}) and (\ref{sign2}) that
\begin{align*}
\sigma^*(\cS(A^t))
&=
\sigma^*(\ev^*_{-m_l}(\cS((\la^{(l)})^t))\circ\cdots\circ
\ev_{-m_1}^*(\cS((\la^{(1)})^t)))\\
&\cong
\ev^*_{m_l}(\cS((\la^{(l)})^t)\otimes\operatorname{sgn})
\circ\dots\circ
\ev^*_{m_1}(\cS((\la^{(1)})^t)\otimes\operatorname{sgn}).
\end{align*}
Since $\cS(\mu^t) \otimes \operatorname{sgn} \cong \cS(\mu)$ for ordinary
representations of the symmetric group, this equals
$\widetilde{\cS}(A)$
according to the definition (\ref{dualspecht}),
which is isomorphic to $\cS(A)^\circledast$ by (\ref{Aredual}).
\end{proof}

Now we can describe what happens when we twist
irreducible $H^{\La^t}_{\alpha^t}$-modules by the automorphism $\sigma$.
In particular, this theorem gives a first explanation for
the alternate parametrization of irreducible modules mentioned
in Remark~\ref{ap}.

\begin{Theorem}\label{mull}
For $A \in \RStd^\La_\alpha$,
we have that $\sigma^*(\cD(A^t)) \cong \widetilde{\cD}(A)
\cong \cD(A^{\downarrow})$.
\end{Theorem}

\begin{proof}
By Theorem~\ref{class} and Lemma~\ref{bij2},
$\cD(A^t)$ is the irreducible head of $\cS(A^t)$.
Hence $\sigma^*(\cD(A^t))$ is the irreducible head
of $\sigma^*(\cS(A^t))$, which is isomorphic
to $\widetilde{\cS}(A)$ according to Lemma~\ref{ssig}.
Now use Remark~\ref{ap}.
\end{proof}

\section{Young modules and tilting modules}

In this section we prove a number of additional results
which complete the Schur-Weyl duality picture from
\cite{schur}. The results in $\S\S$\ref{sark}--\ref{ssym}
should be compared with
Mathas' results from \cite{M}.

\subsection{Schur functors commute with duality}
Recall the dualities $\circledast$ on $\mathcal O^\La$ and
on $\rep{H^\La}$ defined in $\S$\ref{spco} and $\S$\ref{sgir}, respectively. 
We want to prove that these two dualities
are intertwined by the Schur functor
$\pi$ from (\ref{sf1}).
Recall to start with 
from \cite[Theorem A.2]{schur} that $H_d^\bm$
is a symmetric algebra with symmetrizing form
$\tau:H_d^\bm \rightarrow \C$
defined by picking out the
$x_1^{l-1} \cdots x_d^{l-1}$-coefficient of a vector
when expanded in terms of the basis (\ref{bt}).
This restricts to a symmetrizing form on each 
$H^\La_\alpha$
for $\alpha \in Q_+$ with $\height(\alpha) =d$.
Hence each $H^\La_\alpha$ is also a symmetric algebra.

As well as the duality $\circledast$ on $\rep{H_\alpha^\La}$,
there is another duality $\#$
mapping a left module $M$
to
\begin{equation}
M^\# := \hom_{H_\alpha^\La}(M, H_\alpha^\La)
\end{equation}
viewed as a left $H_\alpha^\La$-module with action
$(h f)(m) = f(m) h^*$ for $h \in H_\alpha^\La$, $m \in M$
and $f:M \rightarrow H_\alpha^\La$;
here $h^*$ denotes the image of $h$ under the anti-automorphism
from (\ref{Star}).
By general principles, the fact that $H_\alpha^\La$ 
is symmetric implies that this
duality is equivalent to the duality $\circledast$.

\begin{Lemma}\label{sd2}
For any finite dimensional 
left $H_\alpha^\La$-module $M$, there is a natural $H^\La_\alpha$-module
isomorphism
$M^\# \stackrel{\sim}{\rightarrow} M^\circledast,\: f \mapsto \tau\circ f$.
\end{Lemma}

\begin{proof}
See \cite[Theorem 3.1]{Rick}.
\end{proof}

Now fix $\alpha \in Q_+$ of height $d$ and identify
$H^\La_\alpha$ with $\End_{\mathfrak{g}}(\cT^\La_\alpha)^\op$
according to Theorem~\ref{swd0}.

\begin{Lemma}\label{le2}
There is a non-degenerate symmetric bilinear form $(.,.)$ on
$\cT^\La_\alpha$ such that $(xvh,w) = (v,x^* w h^*)$ for
all $v,w \in \cT^\La_\alpha$, $x \in \mathfrak{g}$ and
$h \in H^\La_\alpha$ (recall $x^*$ denotes the transpose
of the 
matrix $x$).
\end{Lemma}

\begin{proof}
Recalling that 
$\cT^\La_\alpha$ is a block of $\cP(A^\La) \otimes \cT_d$,
it suffices to construct such a non-degenerate form on all
of $\cP(A^\La) \otimes \cT_d$.
As $\cP(A^\La)$ is irreducible, 
it admits a non-degenerate symmetric bilinear form $(.,.)$
such that $(x v, w) = (v, x^* w)$ for each $x \in \mathfrak{g}$
and $v, w \in \cP(A^\La)$ .
There is also a non-degenerate symmetric bilinear form on
$\cT_d$ 
with respect to which the monomials 
in the standard basis of the natural 
$\mathfrak{g}$-module are orthonormal.
The product of these forms gives us a 
non-degenerate symmetric bilinear form on $\cP(A^\La) \otimes \cT_d$
such that $(xv, w) = (v, x^* w)$ again.
The fact that
$(v h, w) = (v, wh^*)$ for each $h \in H_d$
is clear for the action of each $s_r$,
and follows for the action of
$x_1$ by the definition of the form as a product,
recalling $x_1$ acts as multiplication by $\Omega$.
\end{proof}

For the next two lemmas, we define another functor
\begin{equation}
\delta := \hom_{\mathfrak{g}}(?, \cT^\La_\alpha):\mathcal{O}^\La_\alpha
\rightarrow \rep{H^\La_\alpha},
\end{equation}
with $h \in H^\La_\alpha$ acting on $f \in \delta(M) = \hom_{\mathfrak{g}}(M, \cT^\La_\alpha)$ by the rule $(hf)(m) = f(m)h^*$.

\begin{Lemma}\label{le3}
There is an isomorphism $\delta \circ \pi^* \cong \#$ as endofunctors of
$\rep{H^\La_\alpha}$.
\end{Lemma}

\begin{proof}
There are natural isomorphisms
\begin{align*}
\delta \circ \pi^*(M) &= \hom_{\mathfrak{g}}(\cT^\La_\alpha \otimes_{H^\La_\alpha} M, \cT^\La_\alpha)\\&\cong \hom_{H^\La_\alpha}(M, \End_{\mathfrak{g}}(\cT^\La_\alpha)) \cong
\hom_{H^\La_\alpha}(M, H^\La_\alpha) = M^\#,
\end{align*}
for any $M \in \rep{H^\La_\alpha}$.
\end{proof}

\begin{Lemma}\label{le4}
  There is an isomorphism $\delta \cong \pi \circ \circledast$ of functors
from $\mathcal O^\La_\alpha$ to $\rep{H^\La_\alpha}$.
\end{Lemma}

\begin{proof}
Define a natural homomorphism
$\delta(M) \rightarrow \pi(M^\circledast)$ for any $M \in \mathcal O^\La_\alpha$
by mapping $f \in \delta(M) = \hom_{\mathfrak{g}}(M, \cT^\La_\alpha)$
to $\hat f \in \pi(M^\circledast) = \hom_{\mathfrak{g}}(\cT^\La_\alpha, M^\circledast)$, where $\hat f (t)(m)  = (t,f(m))$ for $t \in \cT^\La_\alpha$ and $m \in M$.
Here, $(.,.)$ is the bilinear form from Lemma~\ref{le2}. The non-degeneracy
of this form gives easily that our map is injective. It is an isomorphism
because $\delta(M)$ and $\pi(M^\circledast)$ both have the same dimension.
To see that, note that $M$ and $M^\circledast$ have all the same composition multiplicities, and $\cT^\La_\alpha$ is a direct sum of prinjective indecomposable modules, each of which is both the projective cover and the injective hull of the same irreducible module.
\end{proof}

\begin{Theorem}\label{duality}
There is an isomorphism
$\pi \circ \circledast \cong \circledast \circ \pi$
of functors from $\mathcal O_\alpha^\La$ to $\rep{H_\alpha^\La}$.
\end{Theorem}

\begin{proof}
In view of Lemma~\ref{sd2}, it suffices to show that
$\pi \cong \# \circ \pi \circ \circledast$. By Lemmas~\ref{le3} and \ref{le4},
we have that
$$
\# \circ \pi \circ \circledast \cong \#\circ \delta \cong \delta \circ \pi^* \circ \delta \cong \pi \circ \circledast\circ \pi^* \circ \delta.
$$
So $\pi\circ\circledast\circ\pi^*\circ\delta$ is an exact functor
(because $\#\circ\pi\circ\circledast$ clearly is such), and we need to show that
$\pi \cong \pi\circ \circledast\circ\pi^* \circ \delta$.

Let us first construct a natural $\mathfrak{g}$-module homomorphism
$$
\theta_M: M \rightarrow 
\circledast \circ \pi^* \circ \delta(M) = 
\left[ \cT^\La_\alpha \otimes_{H^\La_\alpha}
\hom_{\mathfrak{g}}(M, \cT^\La_\alpha) \right]^\circledast
$$
for every $M \in \mathcal O^\La_\alpha$,
by letting $\theta_M(m)(t \otimes f) := (t, f(m))$ for $m \in M, t \in
\cT^\La_\alpha$ and $f \in \hom_{\mathfrak{g}}(M, \cT^\La_\alpha)$.
Here $(.,.)$ is the bilinear form from Lemma~\ref{le2} again.
To see that $\theta_M$ is well-defined, note that
$$
\theta_M(m)(th \otimes f) = (th, f(m)) = (t, f(m)h^*)
= \theta_M(m)(t \otimes hf).
$$
To see that $\theta_M$ is a $\mathfrak{g}$-module homomorphism, note that
$$
\theta_M(xm)(t \otimes f) = (t, f(xm)) = (t, xf(m)) = (x^* t, f(m))
= (x \theta_M(m))(t \otimes f).
$$
Observe finally that $\theta_M$ is an isomorphism
in the special case that $M = \cT^\La_\alpha$.
Indeed, in this case, $\cT^\La_\alpha \otimes_{H^\La_\alpha}
\hom_{\mathfrak{g}}(M, \cT^\La_\alpha)$ can be identified simply with
$\cT^\La_\alpha$, and $\theta_M$ is the isomorphism
$\cT^\La_\alpha \stackrel{\sim}{\rightarrow} (\cT^\La_\alpha)^\circledast$
induced by the non-degenerate bilinear form $(.,.)$.

Applying the functor $\pi$ to the homomorphism $\theta_M$ from the previous
paragraph, we obtain a natural $H^\La_\alpha$-module homomorphism
$$
\eta_M:\pi(M) \rightarrow \pi \circ \circledast \circ \pi^* \circ \delta(M),
\quad f \mapsto \theta_M \circ f
$$
for $M \in \mathcal O^\La_\alpha$.
We have now defined a natural transformation
$\eta:\pi \rightarrow \pi \circ \circledast \circ \pi^* \circ \delta$
between two exact functors. 
To complete the proof of the theorem, we need to show that this natural 
transformation is an isomorphism.
Taking $M = \cT^\La_\alpha$, we get that $\eta_M$
is an isomorphism, because $\theta_M$ is already an isomorphism in that case by the last statement of the previous paragraph.
Hence, by Lemma~\ref{pg} and naturality, $\eta_M$ is an isomorphism whenever $M$ is a prinjective
module.
Moreover by the proof of \cite[Theorem 6.10]{schur},
every projective
module $M \in \mathcal O^\La_\alpha$ 
has a two step resolution $0 \rightarrow M \rightarrow I_1 \rightarrow I_2$
where $I_1$ and $I_2$ are prinjective modules.
Using this and the five lemma, we get that $\eta_M$ is an isomorphism
whenever $M$ is a projective module in $\mathcal O^\La_\alpha$.
Finally for an arbitrary $M$ we take a two step projective resolution
$P_2 \rightarrow P_1 \rightarrow M \rightarrow 0$
and apply the five lemma once more.
\end{proof}

\subsection{Young modules and the double centralizer property}\label{sym}
We next introduce another important family of $H^\La_\alpha$-modules,
the so-called {\em Young modules} $\cY(A)$ for $A \in \Col^\La_\alpha$.
For $I$ bounded-below, these are simply the images of the
projective indecomposable modules $\cP(A)$ in $\mathcal O^\La_\alpha$
under the Schur functor $\pi$. We will give a
more intrinsic definition which is valid for arbitrary $I$.

So assume that $I$ is arbitrary and $\La$ and $\alpha$ are 
fixed as usual.
Given $A \in \Col_\alpha^\La$, set $\la  = (\la^{(1)},\dots,\la^{(l)}):= 
\la(A)$ and 
$d_i := |\la^{(i)}|$ for each $i=1,\dots,l$.
Note that $d = d_1 + \dots + d_l$ is the height of $\alpha$.
Let $S_{\la^{(i)}}$ denote the parabolic subgroup 
$S_{\la^{(i)}_1} \times S_{\la^{(i)}_2} \times \cdots$
of the symmetric group 
$S_{d_i}$,
so that $S_{\la^{(1)}} \times \cdots \times S_{\la^{(l)}}$ is a parabolic
subgroup of $S_d$.
Define $\cX(A)$ to be the left ideal of $H^\La_\alpha$ generated
by the element
\begin{equation}\label{xla}
\prod_{i=2}^l \prod_{j=1}^{d_1+\cdots+d_{i-1}}
(x_j - m_i) \cdot \sum_{w \in S_{\la^{(1)}} \times \cdots \times S_{\la^{(l)}}} w.
\end{equation}
We call $\cX(A)$ the {\em reduced permutation module}
because it is a block-wise version of (the degenerate analogue of) the
permutation module of Dipper, James and Mathas
\cite{DJM}.
More precisely, the Dipper-James-Mathas
permutation module 
is the left ideal of $H_d^\bm$ generated by
the element (\ref{xla}),
and then our $\cX(A)$ is obtained from that by multiplying by the
block idempotent $e_{\alpha}$.

To connect the reduced permutation module $\cX(A)$
to parabolic category $\mathcal O$, assume in this paragraph
that $I$ is bounded-below, so that the Schur functor
$\pi$ from (\ref{sf1}) is defined.
Recall the {\em divided power module}
$\cZ(A)$ from \cite[(4.7)]{schur}:
\begin{equation}\label{dpm}
\cZ(A) := U(\mathfrak{g}) \otimes_{U(\mathfrak{p})}
\left[
I^\La \otimes Z^{\la^{(1)}}(V_1) \otimes\cdots\otimes Z^{\la^{(l)}}(V_l)
\right].
\end{equation}
Some explanations are needed here.
First $\mathfrak{p}$ is the standard parabolic subalgebra of
$\mathfrak{g}$ with Levi factor
$\mathfrak{gl}_{n_1}(\C)\oplus\cdots\oplus \mathfrak{gl}_{n_l}(\C)$
as in $\S$\ref{spco}.
Then $I^\La$ denotes the one dimensional $\mathfrak{p}$-module
associated to the weight
$$
\sum_{i=1}^l (n_1+\cdots+n_{i-1}+m_i)(\eps_{n_1+\cdots+n_{i-1}+1}+\cdots+\eps_{n_1+\cdots+n_i}).
$$
Next $V_i$ denotes the $\mathfrak{p}$-submodule of
the natural $\mathfrak{g}$-module of column vectors spanned by
the first $n_1+\cdots+n_i$ of the standard basis vectors.
Finally for a vector space $E$ and a partition $\mu$ of $n$
we write $Z^\mu(E)$ for the submodule of $E^{\otimes n}$
consisting of all tensors that are invariant with respect to the
natural action of the parabolic subgroup $S_\mu$ of $S_n$.
According to \cite[Theorem 4.14]{schur}, we have that
\begin{equation}\label{y1}
\cZ(A) = \cP(A) \oplus (*)
\end{equation}
where $(*)$ is a direct sum of $\cP(B)$'s for $B \in \Col^\La$
that are higher than $A$ in the sense that
$\la(B) > \la(A)$ in the dominance order on multipartitions,
i.e. the diagram of $\la(B)$ is 
obtained from that of 
$\la(A)$ by moving boxes up.
We also need the following key isomorphism which is established
in \cite[Theorem 6.9]{schur}:
for every $A \in \Col^\La_\alpha$ there is an isomorphism
\begin{equation}\label{y2}
\cX(A) \cong \hom_{\mathfrak{g}}(\cT^\La_\alpha, \cZ(A))
\end{equation}
of $H^\La_\alpha$-modules.
Now we can formulate the following theorem which gives the intrinsic
definition of Young modules.

\begin{Theorem}\label{young}
For any index set $I$ and $A \in \Col^\La_\alpha$, there exists a unique
(up to isomorphism) indecomposable $H^\La_\alpha$-module
$\cY(A)$ such that
$$
\cX(A) = \cY(A) \oplus (\dagger)
$$
where $(\dagger)$ denotes a direct sum of $\cY(B)$'s for $B \in \Col^\La_\alpha$
with $\la(B) > \la(A)$ in the dominance order on multipartitions.
Moreover:
\begin{itemize}
\item[(i)]
If $A \in \Std^\La_\alpha$ then $\cY(A)$ coincides with the projective
cover of  the irreducible $H^\La_\alpha$-module $\cD(A)$
from Theorem~\ref{class}.
\item[(ii)]
If $I$ is bounded-below then $\cY(A) \cong \pi(\cP(A))$
for each $A \in \Col^\La_\alpha$.
\end{itemize}
\end{Theorem}

\begin{proof}
Since $\cX(A)$ does not depend on the particular choice of $I$,
we may as well assume that $I$ is bounded-below.
In that case we {\em define} 
$\cY(A)$ to be the $H^\La_\alpha$-module
$\pi(\cP(A))$ for each $A \in \Col^\La_\alpha$, so that (ii) is automatic.
By Theorem~\ref{swd}(ii) this is the projective cover of
$\cD(A)$ whenever $A \in \Std^\La_\alpha$, giving (i).
Moreover by (\ref{y2}) we have that
$\cX(A) = e_{\alpha} \pi(\cZ(A))$, hence by (\ref{y1}) we see that
$$
\cX(A) = \cY(A)+(\dagger)
$$
where $(\dagger)$ consists of higher $\cY(B)$'s 
as in the statement of the theorem.
It remains to observe that $\cY(A)$ is indecomposable.
Since $\cP(A)$ is projective and indecomposable,
we get by Theorem~\ref{swd}(vi) and Fitting's lemma
that
$$
\End_{H^\La_\alpha}(\cY(A)) \cong
\End_{\mathfrak{g}}(\cP(A))
$$
is a local ring. Hence $\cY(A)$ is indecomposable as required.
\end{proof}

The main reason Young modules are important is the
following theorem, which we view as a 
``double centralizer property''.
By a {\em Young generator} for $H^\La_\alpha$, we mean a finite dimensional
module
that is isomorphic to
a direct sum of 
the Young modules $\cY(A)$ for $A \in \Col^\La_\alpha$,
with each appearing at least once.

\begin{Theorem}\label{kos}
Assume that $I$ is bounded-below and $\alpha \in Q_+$.
Let $\cY$ be a {Young generator} for
$H^\La_\alpha$. Then the category of finite dimensional left modules over the
endomorphism algebra
$\End_{H^\La_\alpha}(\cY)^{\op}$
is equivalent to the category $\mathcal O^\La_\alpha$.
\end{Theorem}

\begin{proof}
By Theorem~\ref{young}(ii),
there is a projective generator $\cP$ for $\mathcal O^\La_\alpha$
such that $\cY \cong \pi(\cP)$.
By general principles, the category of finite dimensional
left modules over
$\End_{\mathfrak{g}}(\cP)^\op$
is equivalent to $\mathcal O^\La_\alpha$.
By Theorem~\ref{swd}(vi) the Schur functor $\pi$ is fully faithful on
projectives, so $\pi$ defines an algebra isomorphism
\begin{equation}\label{myg}
\End_{\mathfrak{g}}(\cP)^\op
\cong 
\End_{H^\La_\alpha}(\cY)^\op.
\end{equation}
The theorem follows.
\end{proof}

\begin{Corollary}\label{koszul}
If $\cY$ is a Young generator for $H^\La_\alpha$
then the finite dimensional algebra
$\End_{H^\La_\alpha}(\cY)^{\op}$
admits a unique (up to automorphism) grading
with respect to which it is a Koszul algebra.
\end{Corollary}

\begin{proof}
This follows from the theorem by 
 general results of Beilinson, Ginzburg and Soergel
\cite[Theorem 1.1.3]{BGS} and Backelin \cite[Theorem 1.1]{Back}; see 
also \cite[Corollary 2.5.2]{BGS} for the unicity.
\end{proof}

\begin{Remark}\rm
In particular the (degenerate)
cyclotomic Schur algebra of Dipper, James and Mathas over the ground field $\C$
as defined in \cite[$\S$6]{AMR}
is a Koszul algebra.
This follows from Corollary~\ref{koszul} because 
the blocks of the cyclotomic Schur algebra 
are endomorphism algebras
of Young generators of corresponding $H^\La_\alpha$'s; see \cite[Theorem 6.10]{schur}.
\end{Remark}

\subsection{Tilting modules 
and the opposite Schur-Weyl duality}\label{sark}
We assume in this subsection that the index set $I$ is bounded-below.
As well as the parabolic category $\mathcal O^\La$, we also introduced the
opposite parabolic category $\widetilde{\mathcal O}^\La$ in $\S$\ref{spco}.
There is a precise connection between these two categories provided by
{\em Arkhipov's twisting functor} associated to the longest element $w_0$
of the symmetric group $S_n$. In particular this functor is the key to
the proof of the Arkhipov-Soergel reciprocity formula recorded already in
(\ref{asr}).

To give some more details, let $S_{w_0}$ be Arkhipov's semi-regular bimodule
as in \cite[Theorem 1.3]{Skipp}, which is the 
$(U(\mathfrak{g}), U(\mathfrak{g}))$-bimodule associated to the longest element
$w_0$ in the setup of \cite[p.684]{ASt}.
Then Arkipov's twisting functor
\begin{equation}
\Tw_{w_0}:\widetilde{\mathcal O}^\La \rightarrow {\mathcal O}^\La
\end{equation}
is the right exact functor defined by first
tensoring with $S_{w_0}$ then twisting by the automorphism
$\phi_{w_0}:\mathfrak{g} \rightarrow \mathfrak{g}$ arising from
conjugating by the longest element of the symmetric group $S_n$
(viewed as a permutation matrix):
\begin{equation}
\Tw_{w_0}(M) := \phi_{w_0}^* (S_{w_0} \otimes_{U(\mathfrak{g})} M).
\end{equation}
The key properties of this functor are summarized in the following
known lemma;
note in particular that (\ref{asr}) follows at once from this 
and the self-duality of tilting modules.
In the statement of the lemma, we say that an object
in $\widetilde{\mathcal O}^\La$ 
(resp.\ $\mathcal O^\La$) has a standard 
(resp.\ costandard) filtration if it admits
a filtration whose sections are of the form $\widetilde{\cM}(A)$
(resp.\ $\cM(A)^\circledast$)
for $A \in \Col^\La$.

\begin{Lemma}\label{twist}
The twisting functor $\Tw_{w_0}$ defines an equivalence between
the full subcategory of $\widetilde{\mathcal O}^\La$ consisting of all modules
possessing a standard filtration
and the full subcategory of ${\mathcal O}^\La$ consisting of all modules
possessing a costandard filtration. 
Moreover for $A \in \Col^\La$ we have that
\begin{itemize}
\item[(i)]
$\Tw_{w_0}(\widetilde{\cM}(A)) \cong {\cM}(A)^\circledast$;
\item[(ii)]
$\Tw_{w_0}(\widetilde{\cP}(A)) \cong \cT(A)$.
\end{itemize}
\end{Lemma}

\begin{proof}
This is a slight reformulation of 
\cite[Theorem 6.6]{Skipp} in our special case.
\end{proof}

We also need the following lemma showing that 
the functor $\Tw_{w_0}$ commutes nicely with tensoring with finite
dimensional modules.

\begin{Lemma}\label{tfc}
There
is a natural $\mathfrak{g}$-module isomorphism
$$
j_{M, X}: \Tw_{w_0}(M \otimes X) \stackrel{\sim}{\rightarrow} 
\Tw_{w_0}(M) \otimes X
$$
for every $\mathfrak{g}$-module $M$ and every
finite dimensional $\mathfrak{g}$-module $X$.
This isomorphism has the following additional properties.
\begin{itemize}
\item[(i)] Given also a finite dimensional $\mathfrak{g}$-module $Y$, 
the following diagram commutes:
\begin{equation*}
\begin{CD}
\Tw_{w_0}(M \otimes X \otimes Y)
&@>j_{M\otimes X, Y}>>&\Tw_{w_0}(M \otimes X) \otimes Y\\
@Vj_{M, X \otimes Y}VV&&@VVj_{M,X} \otimes \id_{Y} V\\
\Tw_{w_0}(M) \otimes X \otimes Y&@=&\Tw_{w_0}(M) \otimes X \otimes Y
\end{CD}
\end{equation*}
\item[(ii)]
The isomorphism $j_{M,X}$ commutes with the action of
$z \in Z(U(\mathfrak{g}))$ in the sense that
$j_{M,X} \circ \Tw_{w_0}(\la_z) = \la_z \circ j_{M,X}$,
where $\la_z$ denotes the $\mathfrak{g}$-module endomorphism defined by
left multiplication by $z$.
\end{itemize}
\end{Lemma}

\begin{proof}
All of this except for part (ii) follow at once from \cite[Theorem 3.2]{ASt}.
To deduce (ii), recall from \cite{ASt} that the functor
$\Tw_{w_0}$ is a composite of functors $\Tw_s$ for simple reflections
$s$ (taken in order corresponding to a reduced expression for $w_0$),
and then 
the isomorphism $j_{M,X}$ above is built from analagous isomorphisms
$$
j^s_{M,X}: \Tw_s(M \otimes X) \stackrel{\sim}{\rightarrow} \Tw_s(M) \otimes X
$$
for each simple reflection $s$. In view of this, it 
suffices to show that 
$j^s_{M,X} \circ \Tw_{s}(\lambda_z) = \lambda_z \circ j^s_{M,X}$
for each $s$. The explicit formula for $j^s_{M,X}$
from the proof of \cite[Theorem 3.2]{ASt} gives 
that 
$$
j^s_{M,X}(u (1 \otimes (m \otimes e)))
=\widetilde{\Delta}(u) ((1 \otimes m) \otimes e)
$$
for $u \in U_{(s)}, m \in M$ and $e \in X$.
Here, we are using the notation
from \cite{ASt}; in particular, $U_{(s)}$ is a certain Ore 
localization of $U$
and $\widetilde{\Delta}:U_{(s)} \rightarrow U_{(s)} \hat\otimes U_{(s)}$
is a certain homomorphism
extending the comultiplication $\Delta:U \rightarrow U \otimes U$.
Using this formula, 
the problem reduces to checking that
$$
j^s_{M,X}(u (1 \otimes z(m \otimes e)))
= 
\widetilde{\Delta}(u) \Delta(z) ((1 \otimes m) \otimes e)
$$
is equal to
$$
z j^s_{M,X}(u(1 \otimes (m \otimes e)))
= \Delta(z)\widetilde{\Delta}(u)((1 \otimes m) \otimes e).
$$
To see this, observe that
$$
\widetilde{\Delta}(u) \Delta(z)
= \widetilde{\Delta}(uz) = \widetilde{\Delta}(zu) = \Delta(z) \widetilde{\Delta}(u)
$$
as $z$ is central.
\end{proof}

Now we want to consider the analogue of the Schur-Weyl duality
from $\S$\ref{sswd} with the category $\mathcal O^\La$ replaced by
$\widetilde{\mathcal O}^\La$.
As in the first paragraph of
$\S$\ref{sswd}, there is a natural right action of the
degenerate affine Hecke algebra $H_d$
on the module $\widetilde{\cP}(A^\La) \otimes \cT_d$ commuting with the left action of $\mathfrak{g}$.
For $\alpha \in Q_+$ with $\height(\alpha) = d$, 
define $\widetilde{\cT}^\La_\alpha$ to be the projection
of $\widetilde{\cP}(A^\La) \otimes \cT_d$
onto the block $\widetilde{\mathcal O}^\La_\alpha$. 
This is again
a $(U(\mathfrak{g}), H_d)$-bimodule.

\begin{Lemma}\label{ph2}
The modules $\{\widetilde{\cP}(A)\:|\:A \in \RStd^\La_\alpha\}$
give a full set of prinjective indecomposable modules
in $\widetilde{\mathcal O}^\La_\alpha$.
Moreover,
$\widetilde{\cT}^\La_\alpha$ is a prinjective generator for 
$\widetilde{\mathcal O}^\La_\alpha$.
\end{Lemma}

\begin{proof}
Mimic the proofs of Lemmas~\ref{pg0} and \ref{pg}
(i.e. 
\cite[Theorem 4.8]{schur} and \cite[Corollary 4.6]{schur}), using the reverse
crystal structure from $\S$\ref{stv} instead of the usual crystal structure.
\end{proof}

The following lemma 
gives the analogue of Theorem~\ref{swd0}
in the twisted setup; see also
\cite[Lemma 5.5]{cyclo} for a different approach yielding a slightly
more general result.

\begin{Theorem}\label{is}
There is a $\mathfrak{g}$-module isomorphism
$j_\alpha: \Tw_{w_0}(\widetilde{\cT}^\La_\alpha)\stackrel{\sim}{\rightarrow}
\cT^\La_\alpha$ such that the following diagram
of algebra homomorphisms commutes:
$$
\begin{CD}
&\!\!\!H_d &\\
\\
\End_{\mathfrak{g}}(\widetilde{\cT}^\La_\alpha)^\op
@>\sim >> &\End_{\mathfrak{g}}(\cT^\La_\alpha)^\op
\end{CD}
\begin{picture}(0,0)
\put(-111,2){\makebox(0,0){$\swarrow$}}
\put(-111,2){\line(1,1){18}}
\put(-47,2){\makebox(0,0){$\searrow$}}
\put(-47,2){\line(-1,1){18}}
\end{picture}
$$
where the vertical maps come from the $H_d$-actions
on $\widetilde{\cT}^\La_\alpha$ and on $\cT^\La_\alpha$, respectively,
and the bottom map is the
algebra isomorphism
$\theta \mapsto j_\alpha \circ \Tw_{w_0}(\theta) \circ j_\alpha^{-1}$.
Hence, the action of $H_d$ on $\widetilde{\cT}^\La_\alpha$
induces a canonical isomorphism between $H^\La_\alpha$
and $\End_{\mathfrak{g}}(\widetilde{\cT}^\La_\alpha)^\op$.
\end{Theorem}

\begin{proof}
Set $M := \widetilde{\cP}(A^\La)$ for short.
Consider the isomorphism
$$
j_{M,\cT_d}:\Tw_{w_0}(M \otimes \cT_d)
\stackrel{\sim}{\rightarrow}
\Tw_{w_0}(M) \otimes \cT_d
$$
from Lemma~\ref{tfc}.
We have natural actions of $H_d$ on $\Tw_{w_0}(M \otimes \cT_d)$ 
and on 
$\Tw_{w_0}(M) \otimes \cT_d$;
the former means the action 
obtained by applying the functor $\Tw_{w_0}$ to the natural action of
$H_d$ on
$M \otimes \cT_d$.
We claim that the isomorphism
$j_{M,\cT_d}$ intertwines these two actions.
This is clear for the actions of each $w \in S_d$ 
by the naturality of $j$.
It remains to see that $j_{M, \cT_d}$ intertwines the two actions of $x_1$.
Note that the following diagram commutes by a special case of
Lemma~\ref{tfc}(i):
\begin{equation*}
\begin{CD}
\Tw_{w_0}(M \otimes \cT_d)
&@>j_{M\otimes \cT_1, \cT_{d-1}}>>&\Tw_{w_0}(M \otimes \cT_1) \otimes \cT_{d-1}\\
@Vj_{M, \cT_d}VV&&@VVj_{M,\cT_1} \otimes \id_{\cT_{d-1}} V\\
\Tw_{w_0}(M) \otimes \cT_d&@=&\Tw_{w_0}(M) \otimes \cT_1 \otimes \cT_{d-1},
\end{CD}
\end{equation*}
Using this, our problem reduces to checking
that
$j_{M,\cT_1}$ intertwines the two actions of $x_1$,
i.e. we may assume that $d=1$.
In that case, $x_1$ is defined simply by multiplication by
$\Omega \in U(\mathfrak{g}) \otimes U(\mathfrak{g})$.
Now we note that $\Omega$ can be written as 
$\Delta(z) + z' \otimes 1 + 1 \otimes z''$
for $z, z', z''\in Z(U(\mathfrak{g}))$.
The naturality of $j$ implies at once that 
$j_{M, \cT_1}$ intertwines the endomorphisms arising by multiplication by
 $z' \otimes 1$
and $1 \otimes z''$, so it remains to show that 
$j_{M, \cT_1}$ intertwines the endomorphisms arising by multiplication
by $\Delta(z)$. This follows from Lemma~\ref{tfc}(ii),
establishing the claim.

Now to prove the theorem, we observe by Lemma~\ref{twist} that
$\Tw_{w_0}(M) \cong \cP(A^\La)$. Composing the isomorphism $j_{M, \cT_d}$
with any such isomorphism, we get a $\mathfrak{g}$-module
isomorphism
$\Tw_{w_0}(\widetilde{\cP}(A^\La) \otimes 
\cT_d) \stackrel{\sim}{\rightarrow} \cP(A^\La) 
\otimes \cT_d$
intertwining the two actions of $H_d$. This restricts to give 
the desired isomorphism
$j_\alpha: \Tw_{w_0}(\widetilde{\cT}^\La_\alpha) \stackrel{\sim}{\rightarrow} 
\cT^\La_\alpha$.
Now consider the diagram from the statement of the lemma.
In view of Lemma~\ref{twist}, the functor
$\Tw_{w_0}$ defines an isomorphism
between 
$\End_{\mathfrak{g}}(\widetilde{\cT}^\La_\alpha)^{\op}$
and
$\End_{\mathfrak{g}}(\Tw_{w_0}(\widetilde{\cT}^\La_\alpha))^{\op}$,
which is why the bottom map in the diagram is an isomorphism.
Now the fact that the diagram commutes 
follows because $j_\alpha$ 
intertwines the two actions of $H_d$.
Finally, the last statement of the theorem
follows directly from Theorem~\ref{swd0}.
\end{proof}

Theorem~\ref{is} (and Theorem~\ref{swd0}) allows us henceforth to 
identify
\begin{equation}
H^\La_\alpha = \End_{\mathfrak{g}}(\cT^\La_\alpha)^{\op}
= \End_{\mathfrak{g}}(\widetilde{\cT}^\La_\alpha)^{\op}.
\end{equation}
Recall now the Schur functor $\pi$ from (\ref{sf1}).
Analogously, we introduce the functor
\begin{equation}\label{sf3}
\widetilde{\pi}:= \hom_{\mathfrak{g}}(\widetilde{\cT}^\La_\alpha,
?):\widetilde{\mathcal O}^\La_\alpha
\rightarrow \rep{H^\La_\alpha}
\end{equation}
for any $\alpha \in Q_+$.
The following lemma is the basic tool needed to deduce the main properties
about this twisted Schur functor
from the analogous properties involving $\pi$ from Theorem~\ref{swd}.

\begin{Lemma}\label{snow}
For $j_\alpha$ as in Theorem~\ref{is} and any $M \in \widetilde{\mathcal O}^\La_\alpha$, there is a natural $H^\La_\alpha$-module homomorphism
$\gamma_M:\widetilde{\pi}(M)
\rightarrow
\pi (\Tw_{w_0}(M))$,
such that $\gamma_M(f) = \Tw_{w_0}(f) \circ j_\alpha^{-1}$
for each $f \in \widetilde{\pi}(M)$.
If $M$ has a standard filtration then $\gamma_M$ is an isomorphism.
\end{Lemma}

\begin{proof}
Take $h \in H^\La_\alpha$, and let $\theta_h$
(resp.\ $\widetilde{\theta}_h$) denote the endomorphism
of $\cT^\La_\alpha$ (resp.\ $\widetilde{\cT}^\La_\alpha$) defined
by right multiplication by $h$.
For any $f \in \widetilde{\pi}(M) = \hom_{\mathfrak{g}}(\widetilde{\cT}^\La_\alpha, M)$, we have that 
$hf = f \circ \widetilde{\theta}_h$.
Invoking Theorem~\ref{is} for the last equality, we get that
\begin{align*}
\gamma_M(hf) &= \gamma_M(f \circ \widetilde{\theta}_h)
=
\Tw_{w_0}(f \circ \widetilde{\theta}_h) \circ j_\alpha^{-1}\\
&= 
\Tw_{w_0}(f) \circ j_\alpha^{-1} \circ j_{\alpha} \circ \Tw_{w_0}(\widetilde{\theta}_h) \circ j_\alpha^{-1}=
\gamma_M(f) \circ \theta_h = h \gamma_M(f).
\end{align*}
This proves that $\gamma_M$ is an $H^\La_\alpha$-module homomorphism,
and it is clearly natural in $M$.
Finally, suppose that $M$ has a standard filtration.
As $\widetilde{\cT}^\La_\alpha$ has a standard filtration too (e.g. by
Lemma~\ref{ph2}),
Lemma~\ref{twist} implies that
the map $$
\hom_{\mathfrak{g}}(\widetilde{\cT}^\La_\alpha, M)
\rightarrow \hom_{\mathfrak{g}}(\Tw_{w_0}(\widetilde{\cT}^\La_\alpha),
\Tw_{w_0}(M))$$ 
defined by applying the functor $\Tw_{w_0}$ is an isomorphism.
Hence the map $\gamma_M$ is an isomorphism in this case.
\end{proof}

The following theorem gives analogues of some parts of Theorem~\ref{swd}
in the twisted setup; all the other parts have obvious
analogues too.
Recall the definitions of
$\widetilde{\cS}(A)$ and $\widetilde{\cD}(A)$
from (\ref{dualspecht}) and (\ref{trans}).

\begin{Theorem}\label{swdnew}
The following hold for $A \in \Col^\La_\alpha$:
\begin{itemize}
\item[(i)] $\widetilde{\pi}(\widetilde{\cM}(A))
\cong \widetilde{\cS}(A) \cong \cS(A)^\circledast$;
\item[(ii)] 
$\widetilde{\pi}(\widetilde\cL(A))$ is non-zero
if and only if $A \in \RStd_\alpha^\La$,
in which case
we have that $\widetilde{\pi}(\widetilde\cL(A)) 
\cong \widetilde{\cD}(A)
\cong \cD(A^{\downarrow})$.
\item[(iii)] $\widetilde{\pi}(\widetilde{\cP}(A))
\cong \pi(\cT(A))$, which is isomorphic
to $\cY(A^{\downarrow})$ in case $A \in \RStd^\La_\alpha$.
\end{itemize}
\end{Theorem}

\begin{proof}
By Lemma~\ref{snow}, Lemma~\ref{twist}(i), Theorem~\ref{duality},
Theorem~\ref{swd}(iii) and (\ref{Aredual}),
we have that
$$
\widetilde{\pi}(\widetilde{\cM}(A))
\cong \pi(\Tw_{w_0}(\widetilde{\cM}(A)))
\cong \pi(\cM(A)^\circledast)
\cong \pi(\cM(A))^\circledast
\cong \cS(A)^\circledast
\cong \widetilde{\cS}(A),
$$
giving (i).
For (iii), we have by Lemma~\ref{snow} and Lemma~\ref{twist}(ii)
that
$$
\widetilde{\pi}(\widetilde{\cP}(A))
\cong \pi(\Tw_{w_0}(\widetilde{\cP}(A)))
\cong \pi(\cT(A)).
$$
Moreover if $A \in \RStd^\La_\al$ 
then $\cT(A) \cong \cP(A^{\downarrow})$ by Lemma~\ref{soc},
hence $\pi(\cT(A)) \cong \cY(A^{\downarrow})$ by Theorem~\ref{swd}(ii).
Finally to deduce (ii), note from Lemma~\ref{ph2}
that $\widetilde{\pi}(\widetilde{\cL}(A))$ is non-zero
if and only if 
$A \in \RStd^\La_\alpha$.
Moreover, assuming $A \in \RStd^\La_\alpha$,
we get that $\widetilde{\pi}(\widetilde{\cL}(A))$
is an irreducible $H^\La_\alpha$-module
exactly
as in Theorem~\ref{swd}(i), and
$\widetilde{\pi}(\widetilde{\cP}(A)) \cong \cY(A^{\downarrow})$ 
is its projective cover
exactly as in Theorem~\ref{swd}(ii).
Since $\cY(A^{\downarrow})$ is the projective cover of
$\cD(A^{\downarrow})$, we deduce that
$\widetilde{\pi}(\widetilde{\cL}(A))
\cong \cD(A^{\downarrow})$, which is isomorphic to
$\widetilde{\cD}(A)$ by (\ref{trans}).
\end{proof}

\subsection{Signed Young modules}\label{ssym}
For completeness, we want finally to explain briefly 
how to identify the modules 
in Theorem~\ref{swdnew}(iii) with signed Young modules.
We omit some of the details in the proofs here.
Assuming to start with that $I$ is arbitrary,
take any $A \in \Col^\La_\alpha$, and let
$\widetilde\cX(A)$ denote the left ideal of $H^\La_\alpha$
generated by the element
\begin{equation}\label{wt4}
\prod_{i=1}^{l-1}
\prod_{j=1}^{d_{i+1}+\cdots+d_l}
(x_j-m_i) \cdot
\sum_{w \in S_{(\la^{(l)})^t} \times\cdots\times S_{(\la^{(1)})^t}}
\operatorname{sgn}(w) w,
\end{equation}
where $\la  = (\la^{(1)},\dots,\la^{(l)}) 
:= \la(A)$ and $d_i := |\la^{(i)}|$.
This is the {\em reduced signed permutation module} as in
\cite[$\S$4]{M}.
The next lemma explains how to define the {\em signed Young module} $\widetilde\cY(A)$
 as a summand of $\widetilde\cX(A)$.

\begin{Lemma}\label{wdz}
There exists a unique (up to isomorphism) indecomposable
$H^\La_\alpha$-module $\widetilde\cY(A)$ such that
$$
\widetilde{\cX}(A) = \widetilde{\cY}(A) + (\ddagger)
$$
where $(\ddagger)$ denotes a direct sum of $\widetilde{\cY}(B)$'s 
for $B \in \Col^\La_\alpha$
with $\la(B) < \la(A)$ in the dominance ordering on multipartitions.
\end{Lemma}

\begin{proof}
We may assume for the proof that $I = \Z$, so that the automorphism
$\sigma$ from (\ref{sigmalocal}) makes sense.
By (\ref{xla}) and (\ref{wt4}), we have that
$\widetilde{\cX}(A) = \sigma^*(\cX(A^t))$.
Setting
\begin{equation}\label{signedones}
\widetilde{\cY}(A) := \sigma^*(\cY(A^t))
\end{equation}
and noting that the bijection (\ref{bij}) is order-reversing with respect to the
dominance ordering, the lemma 
now follows from the first part of Theorem~\ref{young}.
\end{proof}

For the remainder of the subsection we assume that $I$ is bounded-below.
In the same spirit as (\ref{dpm}),
we define the {\em exterior power module}
\begin{equation}\label{epm}
\cEx(A) := U(\mathfrak{g}) \otimes_{U(\mathfrak{p})}
\left[
I^\La \otimes 
{\textstyle\bigwedge}^{(\la^{(l)})^t}(V^l)
\otimes\cdots\otimes {\textstyle\bigwedge}^{(\la^{(1)})^t} (V^1)\right].
\end{equation}
Here, $V^i$ denotes the quotient of the natural module
of column vectors by the $\mathfrak{p}$-submodule
$V_{i-1}$ defined just after (\ref{dpm}). 
Also for a vector space $E$ and a partition $\mu$ of $n$
we write $\bigwedge^\mu(E)$ for the quotient 
$\bigwedge^{\mu_1}(E) \otimes \bigwedge^{\mu_2}(E) \otimes\cdots$
of $E^{\otimes n}$.
By arguments similar to the proof of 
(\ref{y1}) (which is \cite[Theorem 4.14]{schur}),
one shows that
\begin{equation}\label{hr}
\cEx(A)  = \cT(A) + (\#)
\end{equation}
where $(\#)$ is a direct sum of $\cT(B)$'s for $B \in \Col^\La$
with $\la(B) < \la(A)$ in the dominance order;
the key point is that 
$R(\cT(A') \boxtimes \cT(A'')) \cong \cT(A)$
in the setup and notation of \cite[Corollary 4.12(i)]{schur}.
Moreover it is the case that
\begin{equation}\label{remains}
\widetilde{\cX}(A) \cong \hom_{\mathfrak{g}}(\cT^\La_\alpha, \cEx(A)).
\end{equation}
This is proved by the same techniques used
to prove (\ref{y2}) (which is \cite[Theorem 6.9]{schur}); 
in particular the proof goes via finite $W$-algebras.
(There are some additional complications in the present setting
since this argument realizes $\pi(\cT(A))$ initially as a quotient of
$H^\La_\alpha$, whereas $\widetilde\cY(A)$ is a submodule;
this is overcome by using Theorem~\ref{duality}, the self-duality
of the tilting module $\cEx(A)$, and the symmetric algebra structure on
$H^\La_\alpha$.)

\begin{Theorem}\label{wedf}
$\pi(\cT(A)) \cong \widetilde{\cY}(A)$.
\end{Theorem}

\begin{proof}
This follows from (\ref{hr}), Lemma~\ref{wdz},
Theorem~\ref{swd}(vi) and (\ref{remains}), in exactly the same way that
Theorem~\ref{young}(ii) was deduced from (\ref{y1}) in the proof
of that theorem.
\end{proof}

We say that $\widetilde{\cY}$ is a {\em signed Young generator}
for $H^\La_\alpha$ if it is a finite direct sum of the signed
Young modules $\widetilde{\cY}(A)$ for $A \in \Col^\La_\alpha$
with each appearing at least once.

\begin{Corollary}
If $\widetilde{\cY}$ is a signed Young generator
for $H^\La_\alpha$
 then
the category of finite dimensional left modules over the algebra
$\End_{H^\La_\alpha}(\widetilde{\cY})^{\op}$
is equivalent to the category $\widetilde{O}^\La_\alpha$.
Moreover $\End_{H^\La_\alpha}(\widetilde{\cY})^{\op}$ 
is the Ringel dual of the algebra 
$\End_{H^\La_\alpha}(\cY)^{\op}$ from Theorem~\ref{kos}.
\end{Corollary}

\begin{proof}
By Theorems~\ref{wedf} and \ref{swdnew}(iii), there is a tilting generator
$\cT$ for $\mathcal O^\La_\alpha$
and a projective generator $\widetilde{\cP}$ for $\widetilde{\mathcal O}^\La_\alpha$ such that
$\pi(\cT) \cong \widetilde{\cY}\cong \widetilde{\pi}(\widetilde{\cP})$.
By Theorem~\ref{swd}(vi) and Lemma~\ref{twist}, we get
that
$$
\End_{\mathfrak{g}}(\cT)^{\op} \cong \End_{H^\La_\alpha}(\widetilde{\cY})^{\op}
\cong \End_{\mathfrak{g}}(\widetilde\cP)^{\op}.
$$
Recalling Theorem~\ref{kos},
the 
left hand algebra is the Ringel dual of the algebra 
$\End_{H^\La_\alpha}(\cY)^{\op}$,
while the category of finite dimensional modules over the
right hand algebra is obviously equivalent to 
$\widetilde{\mathcal O}^\La_\alpha$.
\end{proof}

\begin{Corollary}
The following hold
for any $A \in \Col^\La_\alpha$:
\begin{itemize}
\item[(i)] $\cY(A)^\circledast \cong \cY(A)$;
\item[(ii)] $\widetilde{\cY}(A)^\circledast \cong \widetilde{\cY}(A)$.
\end{itemize}
\end{Corollary}

\begin{proof}
We get that $\widetilde{\cY}(A)^\circledast
\cong \widetilde{\cY}(A)$ immediately from Theorem~\ref{wedf},
as the tilting module $\cT(A)$ is self-dual and the Schur functor
$\pi$ commutes with duality according to Theorem~\ref{duality}.
This proves (ii).
Then (i) follows from (ii) by twisting with the automorphism
$\sigma$, using (\ref{sign1}) and (\ref{signedones}).
\end{proof}


\begin{thebibliography}{AMR}
\bibitem[AnS]{ASt}
H. Andersen and C. Stroppel,
Twisting functors on $\mathcal O$,
{\em Represent. Theory} {\bf 7} (2003), 681--699.

\bibitem[AS]{AS}
T. Arakawa and T. Suzuki,
Duality between $\mathfrak{sl}_n(\mathbb{C})$ and the degenerate affine Hecke algebra, {\em  J. Algebra} {\bf  209} (1998), 288--304.

\bibitem[A1]{Ari}
S. Ariki,
On the decomposition numbers of the Hecke algebra of $G(m,1,n)$, 
{\em J. Math. Kyoto Univ.} {\bf 36} (1996), 789--808. 

\bibitem[A2]{Abranch}
S. Ariki,
 Proof of the modular branching rule for cyclotomic Hecke algebras, 
{\em J. Algebra} {\bf 306} (2006), 290--300.

\bibitem[AMR]{AMR}
S. Ariki, 
A. Mathas and H. Rui,
Cyclotomic Nazarov-Wenzl algebras,
{\em Nagoya Math. J.} {\bf 182} (2006),
47--134.

\bibitem[B]{Back}
E. Backelin,
Koszul duality for parabolic and singular category $\mathcal O$,
{\em Represent. Theory} {\bf 3} (1999), 139--152.

\bibitem[BB]{BB}
A. Beilinson and J. Bernstein,
Localisation de $\mathfrak g$-modules,
{\em C. R. Acad. Sci. Paris Ser. I Math.} {\bf 292} (1981), 15--18.

\bibitem[BGS]{BGS}
A. Beilinson, V. Ginzburg and W. Soergel,
Koszul duality patterns in representation theory,
{\em J. Amer. Math. Soc.} {\bf 9} (1996), 473--527.

\bibitem[BFK]{BFK}
J. Bernstein, I. Frenkel and M. Khovanov,
A categorification of the Temperley-Lieb algebra and Schur quotients of
$U(\mathfrak{sl}_2)$ via projective and Zuckerman functors.
{\em Selecta Math. (N.S.)} {\bf 5} (1999), 199--241.

\bibitem[BGG]{BGG}
J. Bernstein, I. M. Gelfand and S. I. Gelfand,
A category of $\mathfrak g$-modules,
{\em Func. Anal. Appl.} {\bf 10} (1976), 87--92.

\bibitem[B1]{dual}
J. Brundan,
Dual canonical bases and Kazhdan-Lusztig polynomials,
{\em J. Algebra} {\bf 306} (2006), 17-46.

\bibitem[B2]{cyclo}
J. Brundan,
Centers of degenerate cyclotomic Hecke algebras and parabolic category
$\mathcal O$, {\em Represent. Theory} {\bf 12} (2008), 236--259.

\bibitem[BDK]{GL}
J. Brundan, R. Dipper and A. Kleshchev,
Quantum linear groups and representations of
$GL_n({\mathbb F}_q)$,
{\em Mem. Amer. Math. Soc.}
{\bf 149} (2001), no. 706, 112 pages.

\bibitem[BK1]{BKtf}
J. Brundan and A. Kleshchev,
Translation functors for general linear and symmetric groups,
{\em Proc. London Math. Soc.}
{\bf 80} (2000), 75--106.

\bibitem[BK2]{rep}
J. Brundan and A. Kleshchev,
Representations of shifted Yangians and finite $W$-algebras,
{\em Mem. Amer. Math. Soc.}
{\bf 196} (2008), no. 918, 107 pages.

\bibitem[BK3]{schur}
J. Brundan and A. Kleshchev,
Schur-Weyl duality for higher levels,
{\em Selecta Math. (N.S.)} {\bf 14} (2008), 1--57.

\bibitem[BK4]{young}
J. Brundan and A. Kleshchev,
Blocks of cyclotomic Hecke algebras and Khovanov-Lauda algebras;
{\tt arXiv:0808.2032}.

\bibitem[BS]{BS2}
J. Brundan and C. Stroppel,
Highest weight categories arising from Khovanov's diagram algebra
III: category $\mathcal O$;
{\tt arXiv:0812.1090}.

\bibitem[BrK]{BrK}
J.-L. Brylinksi and M. Kashiwara,
Kazhdan-Lusztig conjecture and holonomic systems,
{\em Invent. Math.} {\bf 64} (1981), 387--410.

\bibitem[CG]{CG}
N. Chriss and V. Ginzburg,
{\em Representation Theory and Complex Geometry},
Birkh\"auser, 1997.

\bibitem[CR]{CR}
J. Chuang and R. Rouquier,
Derived equivalences for symmetric groups and $\mathfrak{sl}_2$-categorification, {\em Ann. of Math.} {\bf 167} (2008), 245--298.

\bibitem[DJM]{DJM}
R. Dipper, G. James and A. Mathas, 
Cyclotomic $q$-Schur algebras, {\em Math. Z.}  {\bf 229}
(1999), 385--416.

\bibitem[D]{D}
V. Drinfeld, 
Degenerate affine Hecke algebras and Yangians, {\em 
Func. Anal. Appl.} {\bf 20} 
(1986), 56--58.

\bibitem[F]{Fulton}
W. Fulton,
{\em Young Tableaux}, London Math. Soc., 1997.

\bibitem[G]{Gab}
P. Gabriel,
Des cat\'egories Ab\'eliennes,
{\em Bull. Soc. Math. France} {\bf 90} (1962), 323--448.

\bibitem[G1]{Gnote}
I.  Grojnowski, 
Representations of affine Hecke algebras (and affine quantum ${\rm GL}\sb n$) 
at roots of unity,
{\em Internat. Math. Res. Not.} {\bf 5} (1994), 215--217. 

\bibitem[G2]{Groj}
I. Grojnowski, Affine $\widehat{\mathfrak{sl}}_p$ controls the modular
representation theory of the symmetric group and related Hecke algebras;
{\tt math.RT/9907129v1}.

\bibitem[H]{H}
A. Henderson, Nilpotent orbits of linear and cyclic quivers and Kazhdan-Lusztig polynomials of type A, {\em Represent. Theory} {\bf 11} (2007), 95--121. 

\bibitem[J]{Jac}
N. Jacon,
On the parametrization of the simple modules for Ariki-Koike algebras at roots of unity,
{\em
J. Math. Kyoto Univ.} {\bf 44} (2004), 729--767. 

\bibitem[K1]{KaG}
M. Kashiwara, 
Global crystal bases of quantum groups,
{\em Duke Math. J.} {\bf 69} (1993), 455--485.

\bibitem[K2]{Ka}
M. Kashiwara, On crystal bases,
in: ``Representations of Groups,'' {\em CMS Conf. Proc.} {\bf 16}, Amer. Math. Soc., 1995, pp. 155--197.

\bibitem[KN]{KN}
M. Kashiwara and T. Nakashima, Crystal graphs for representations of the $q$-analogue of classical Lie algebras, {\em J. Algebra} {\bf 165} (1994), 295--345. 

\bibitem[KL1]{KL1}
D. Kazhdan and G. Lusztig,
Representations of Coxeter groups and Hecke algebras,
{\em Invent. Math.} {\bf 53} (1979), 165--184. 

\bibitem[KL2]{KL2}
D. Kazhdan and G. Lusztig, 
Proof of the Deligne-Langlands conjecture for Hecke algebras,
{\em Invent. Math.}  {\bf 87} (1987), 153--215. 

\bibitem[K]{Kbook}
A. Kleshchev, {\em
Linear and Projective Representations of Symmetric Groups}, Cambridge University Press, Cambridge, 2005. 

\bibitem[LLT]{LLT}
A. Lascoux, B. Leclerc and J.-Y. Thibon, Hecke algebras at roots of unity and crystal bases of quantum affine algebras, {\em Comm. Math. Phys.} {\bf 181} (1996), 205--263.

\bibitem[LS]{LS}
A. Lascoux and M.-P. Sch\"utzenberger,
Keys and standard bases,
in: ``Invariant Theory and Tableaux'',
D. Stanton ed., Springer 1990.

\bibitem[LT]{LT}
B. Leclerc and J.-Y. Thibon,
Canonical bases of $q$-deformed Fock spaces,
{\em Internat. Math. Res. Notices} {\bf 9} (1996), 447--456. 

\bibitem[L1]{Lc}
G. Lusztig,
Cuspidal local systems and graded Hecke algebras I,
{\em Inst. Hautes \'Etudes Sci. Publ. Math.} {\bf 67}
(1988), 145--202.

\bibitem[L2]{Lu}
G. Lusztig,
Quantum groups at roots of 1,
{\em Geom. Ded.} {\bf 35} (1990), 89--114.

\bibitem[L3]{Lu0}
G. Lusztig,
Quivers, perverse sheaves, and quantized enveloping algebras,
{\em J. Amer. Math. Soc.} {\bf 4} (1991), 365--421.

\bibitem[L4]{Lubook}
G. Lusztig,
{\em Introduction to Quantum Groups}, Birkh\"auser,
1993.

\bibitem[L5]{Ld}
G. Lusztig, 
Cuspidal local systems and graded Hecke algebras, II,
in: ``Representations of Groups,'' {\em CMS Conf. Proc.} {\bf 16}, Amer. Math. Soc., 1995, pp. 217--275.

\bibitem[M]{M}
A. Mathas,
Tilting modules for cyclotomic Schur algebras,
{\em J. reine angew. Math.} {\bf 562} (2003), 137--169.

\bibitem[R]{Rick}
J. Rickard,
Equivalences of derived categories for symmetric algebras,
{\em J. Algebra} {\bf 257} (2002), 460--481.

\bibitem[S]{Skipp}
W. Soergel,
Character formulas for tilting modules over Kac-Moody algebras,
{\em Represent. Theory} {\bf 2} (1998), 432--448. 

\bibitem[VV]{VV}
M. Varagnolo and E. Vasserot, 
On the decomposition matrices of the quantized Schur algebra, 
{\em Duke Math. J.} {\bf 100} (1999), 267--297. 

\bibitem[Z1]{Z1}
A. Zelevinsky,
$p$-adic analogue of the Kazhdan-Lusztig hypothesis,
{\em Funct. Anal. Appl.} {\bf 15} (1981), 83--92.

\bibitem[Z2]{Zel}
A. Zelevinsky,
Two remarks on graded nilpotent orbits,
{\em Russian Math. Surveys} {\bf 40} (1985), 249--250.

\end{thebibliography}
\end{document}